\documentclass[12pt,reqno]{amsart}

\usepackage{amsmath,amssymb,amsthm,bbm,mathtools,calc,verbatim,enumitem,tikz,url,mathrsfs,cite,fullpage,hyperref}
\usepackage{float}
\usepackage{subcaption}

\usepackage{scalerel}
\usepackage{graphicx}
\numberwithin{figure}{section}
\usepackage{comment}

\usepackage{stackengine}
\stackMath
\newcommand\xxrightarrow[2][]{\mathrel{%
  \setbox2=\hbox{\stackon{\scriptstyle#1}{\scriptstyle#2}}%
  \stackunder[1.5pt]{%
    \xrightarrow{\makebox[\dimexpr\wd2\relax]{$\scriptstyle#2$}}%
  }{%
   \scriptstyle#1\,%
  }%
}}

\newcommand\dig[1]{\scalerel*[6.5pt]{\big#1}{%
  \ensurestackMath{\addstackgap[2pt]{\big#1}}}}
\newcommand\digl[1]{\hspace{0.01cm}\mathopen{\dig{#1}}\hspace{0.02cm}}
\newcommand\digr[1]{\hspace{0.03cm}\mathclose{\dig{#1}}}

\newcommand\rig[1]{\scalerel*[6pt]{\big#1}{%
  \ensurestackMath{\addstackgap[3pt]{\big#1}}}}
\newcommand\rigl[1]{\hspace{0.00cm}\mathopen{\rig{#1}}\hspace{0.05cm}}
\newcommand\rigr[1]{\hspace{0.02cm}\mathclose{\rig{#1}}}

\newcommand\sig[1]{\scalerel*[5.5pt]{\big#1}{%
  \ensurestackMath{\addstackgap[1pt]{\big#1}}}}
\newcommand\sigl[1]{\hspace{0.01cm}\mathopen{\sig{#1}}\hspace{0.00cm}}
\newcommand\sigr[1]{\hspace{0.02cm}\mathclose{\sig{#1}}}

\newcommand\pig[1]{\scalerel*[5pt]{\normalsize#1}{%
  \ensurestackMath{\addstackgap[0.5pt]{\normalsize#1}}}}
\newcommand\pigl[1]{\hspace{0.02cm}\mathopen{\pig{#1}}\hspace{0.02cm}}
\newcommand\pigr[1]{\hspace{0.02cm}\mathclose{\pig{#1}}}

\newcommand\gigl[1]{\hspace{0.02cm}\mathopen{\pig{#1}}\hspace{0.01cm}}
\newcommand\gigr[1]{\hspace{0.03cm}\mathclose{\pig{#1}}\hspace{0.01cm}}

\newtheorem{theorem}{Theorem}[section]
\newtheorem{lemma}[theorem]{Lemma}

\newtheorem{conjecture}[theorem]{Conjecture}
\newtheorem{question}[theorem]{Question}
\newtheorem{problem}[theorem]{Problem}

\theoremstyle{definition}
\newtheorem{defn}[theorem]{Definition}

\theoremstyle{remark}

\newcommand\N{\mathbb{N}}
\newcommand\R{\mathbb{R}}

\newcommand\F{\mathbb{F}}

\newcommand\cA{\mathcal{A}}

\newcommand\cC{\mathcal{C}}
\newcommand\cG{\mathcal{G}}
\newcommand\cH{\mathcal{H}}
\newcommand\cI{\mathcal{I}}
\newcommand\cL{\mathcal{L}}

\newcommand\cT{\mathcal{T}}

\def\Pr{\mathbb{P}}
\newcommand\Ex{\mathbb{E}}

	\def\1{\mathbbm{1}}

\newcommand\eps{\varepsilon}

\renewcommand{\geq}{\geqslant}
\renewcommand{\le}{\leqslant}
\renewcommand{\ge}{\geqslant}
\renewcommand{\to}{\rightarrow}

	\def\<{\langle}
	\def\>{\rangle}

\begin{document}

\title{Some recent results in Ramsey theory}

\author{Robert Morris}
\address{IMPA, Estrada Dona Castorina 110, Jardim Bot\^anico, Rio de Janeiro, 22460-320, Brazil}
\email{rob@impa.br}

\thanks{The author is partially supported by CNPq (Proc.~303681/2020-9 and Proc.~407970/2023-1) and by FAPERJ (Proc.~E-26/200.977/2021)}

\begin{abstract}
The purpose of this survey is to provide a gentle introduction to several recent breakthroughs in graph Ramsey theory. In particular, we will outline the proofs (due to various groups of authors) of exponential improvements to the diagonal, near-diagonal, and multicolour Ramsey numbers, improved lower bounds on $R(3,k)$ and $R(4,k)$, and an exponential upper bound on the induced Ramsey numbers. 
\end{abstract}

\maketitle

\section{Introduction}

The Ramsey number $R(k)$ is the smallest $n \in \N$ such that every red-blue colouring of the edges of $K_n$, the complete graph with $n$ vertices, contains a monochromatic copy of $K_k$. These numbers exist by the famous theorem of Ramsey~\cite{R30}, and the bounds
\begin{equation}\label{eq:Rk:ESz:bound}
2^{k/2} \le R(k) \le 4^k
\end{equation}
were proved by Erd\H{o}s and Szekeres~\cite{ESz35} and by Erd\H{o}s~\cite{E47}, whose stunning non-constructive proof of the lower bound initiated the development of the probabilistic method (see~\cite{AS}). 

Over the almost 80 years since these two bounds were proved, the problem of improving either developed into one of the most notorious open questions in combinatorics. Part of the fascination with this problem within the community lies in the fact that it exposes a serious gap in our understanding of `random-like' (or \emph{pseudorandom}) graphs and colourings. 

The study of such colourings has led to some (super-polynomial, but sub-exponential) improvements~\cite{C09,GR,T88,Sah} over the upper bound in~\eqref{eq:Rk:ESz:bound}, as well as to the development of many powerful tools, with a vast array of applications in combinatorics and theoretical computer science (see, e.g.,~\cite{KS}). However, the following theorem, providing an exponential improvement over the upper bound of Erd\H{o}s and Szekeres, was finally proved only a couple of years ago, by Campos, Griffiths, Morris and Sahasrabudhe~\cite{CGMS}.

\begin{theorem}%[Campos, Griffiths, Morris and Sahasrabudhe]
\label{thm:diagonal}
There exists $\eps > 0$ such that 
$$R(k) \le (4 - \eps)^k$$ 
for all sufficiently large $k \in \N$. 
\end{theorem}

The value of $\eps$ obtained in~\cite{CGMS} was quite small, but the approach was later streamlined and optimised by Gupta, Ndiaye, Norin and Wei~\cite{GNNW}, giving $\eps \approx 1/5$. In Section~\ref{diag:sec} we will outline a (significantly shorter) proof of Theorem~\ref{thm:diagonal} that was discovered more recently by the authors of~\cite{CGMS}, together with Balister, Bollobás, Hurley and Tiba~\cite{BBCGHMST}.

%but still differ from the best-known lower bound~\cite{E47,S77} by an exponential factor. 

\subsection{Off-diagonal Ramsey numbers}

The upper bound proved by Erd\H{o}s and Szekeres is actually slightly stronger (by a factor of roughly $\sqrt{k}$) than the one stated in~\eqref{eq:Rk:ESz:bound}. It follows from a simple induction argument, which requires the introduction of the following more general definition. The Ramsey number $R(\ell,k)$ is the smallest $n \in \N$ such that every red-blue colouring of the edges of $K_n$ contains either a red copy of $K_\ell$ or a blue copy of $K_k$. Erd\H{o}s and Szekeres~\cite{ESz35} proved that 
\begin{equation}\label{eq:ESz:bound}
R(\ell,k) \le {k + \ell - 2 \choose \ell - 1}
\end{equation}
for all $\ell,k \in \N$. In particular, setting $\ell = k$ gives $R(k) \le {2k - 2 \choose k-1} \approx \frac{1}{\sqrt{k}} \cdot 4^k$.

Given the difficulty of improving the bounds on the `diagonal' Ramsey numbers $R(k)$, attention partly shifted to understanding the `off-diagonal' Ramsey numbers $R(\ell,k)$, where $\ell$ is fixed and $k \to \infty$. Note that $R(1,k) = 1$ and $R(2,k) = k$, so the bound~\eqref{eq:ESz:bound} is tight in these trivial cases. The first non-trivial case is therefore $R(3,k)$, which turns out to be \emph{much} more interesting, and has been the subject of a huge amount of research over the past 90 years (see~\cite{Sp}). Following groundbreaking work of Erd\H{o}s~\cite{E57,E59,E61} in the 1950s and 1960s, $R(3,k)$ was determined up to a constant factor by Ajtai, Koml\'os and Szemer\'edi~\cite{AKSz81} in 1981, and Kim~\cite{Kim} in 1995. Since then numerous alternative proofs and generalisations %further refinements 
of both the upper~\cite{AKSz,A96,CJMS23,DJPR,DJM,Sh83} and the lower~\cite{Boh,BK,CJMS25,FGM,HHKP} bound have been discovered; in fact, this year alone has seen two significant breakthroughs on the lower bound, in~\cite{CJMS25} and~\cite{HHKP}. 
%a couple of these advances appeared on arXiv just a few weeks before this survey was written, 
As a result of this, the best known bounds on $R(3,k)$ now differ by only a factor of $2 + o(1)$; the upper bound in Theorem~\ref{thm:R3k} was proved by Shearer~\cite{Sh83} in 1983, and the lower bound very recently by Hefty, Horn, King and Pfender~\cite{HHKP}.

\begin{theorem}\label{thm:R3k}
$$\bigg( \frac{1}{2} + o(1) \bigg) \frac{k^2}{\log k} \le R(3,k) \le \big( 1 + o(1) \big) \frac{k^2}{\log k}$$ 
as $k \to \infty$. 
\end{theorem}

In Sections~\ref{sec:R3k:upper} and~\ref{sec:R3k:lower} we will outline the proofs of these two bounds. Very roughly speaking, the approach for the upper bound is to choose the vertices of the blue $K_k$ randomly one by one, and for the lower bound the graph of red edges is formed by the union of two blow-ups of the random graph $G(n,p)$, placed randomly on top of one another. 

Given this success, it is natural to hope that similarly strong bounds can be proved for $R(\ell,k)$ for all fixed $\ell$. Surprisingly, however, while the techniques used to prove Theorem~\ref{thm:R3k} can be applied to give bounds on $R(\ell,k)$, when $\ell \ge 4$ these bounds no longer match, and in fact differ by a (quite large) polynomial factor! More precisely, the lower bound techniques can be extended to prove a bound of the form $R(\ell,k) \ge k^{(\ell+1)/2 + o(1)}$, whereas the best known upper bounds only improve the Erd\H{o}s--Szekeres bound~\eqref{eq:ESz:bound} by a polylogarithmic factor.

Determining which of these bounds is closer to the truth is one of the most important open problems in Ramsey theory, and (as shown in~\cite{MuV}) is closely related to the (conjectured) existence of optimally pseudorandom $K_\ell$-free graphs, see Section~\ref{sec:MuV}. The problem is wide open in general, and when $\ell \ge 5$ we do not know how to improve either of the bounds stated above. However, in an exciting recent breakthrough, the case $\ell = 4$ was resolved (up to poly-logarithmic factors) by Mattheus and Verstraete~\cite{MaV}.

\begin{theorem}\label{thm:R4k}
There exist constants $C,c > 0$ such that
$$\frac{ck^3}{(\log k)^4} \le R(4,k) \le \frac{Ck^3}{(\log k)^2}$$ 
for all sufficiently large $k \in \N$. 
\end{theorem}

The upper bound in Theorem~\ref{thm:R4k} was proved by Ajtai, Koml\'os and Szemer\'edi~\cite{AKSz,AKSz81} in 1980, who showed more generally that 
\begin{equation}\label{eq:AKSz:fixed:ell}
R(\ell,k) \le \frac{Ck^{\ell-1}}{(\log k)^{\ell-2}}
\end{equation}
for each fixed $\ell \ge 3$ and all sufficiently large $k \in \N$ (see Section~\ref{AKSz:sec}). To prove the lower bound, Mattheus and Verstraete used a certain algebraic object known as the Hermitian unital, which provides a collection of roughly $n^{3/4}$ edge-disjoint cliques of size $\sqrt{n}$, with the property (shown by O’Nan~\cite{ONan} in the 1970s) that every copy of $K_4$ in the union of the cliques intersects one of the cliques in (at least) a triangle. To construct a $K_4$-free graph with no large independent set (that is, the red edges of their colouring), they replace each clique by a (random) complete bipartite graph, and then take a random subset of the vertex set of size roughly $n^{3/4}$. In Section~\ref{sec:R4k} we will provide a more detailed outline of their proof. 

\subsection{Ramsey numbers closer to the diagonal}

In the discussion above we restricted our attention to the two extremes: the case $\ell = k$, and the case $\ell$ fixed and $k \to \infty$. However, the method of Ajtai, Koml\'os and Szemer\'edi~\cite{AKSz} can be extended to improve the Erd\H{o}s--Szekeres bound~\eqref{eq:ESz:bound} for all $\ell \ll \log k$, and that of~\cite{CGMS} can be extended to cover the range $\log k \ll \ell \le k$. Neither method covers\footnote{We use standard probabilistic notation, so $f(n) \ll g(n)$ if and only if $f(n)/g(n) \to 0$ as $n \to \infty$.} the range $\ell = \Theta(\log k)$, but fortunately this gap can be filled using an approach due to R\"odl~\cite{GR}. Combining all of these results, we obtain the following exponential improvement over the bound of Erd\H{o}s and Szekeres~\cite{ESz35}. 

\begin{theorem}\label{thm:off:diagonal}
There exists $\delta > 0$ such that 
%\begin{equation}\label{eq:off:diagonal}
$$R(\ell,k) \le e^{-\delta \ell} {k + \ell - 2 \choose \ell - 1}$$
%\end{equation}
for all sufficiently large $k \in \N$, and every $3 \le \ell \le k$. 
\end{theorem}

We will outline the proof of Theorem~\ref{thm:off:diagonal} in Sections~\ref{AKSz:sec}--\ref{diag:sec}, %Sections~\ref{AKSz:sec},~\ref{closer:sec},~\ref{Rodl:sec} and~\ref{diag:sec}, 
with each section covering a different range of $\ell$. In particular, in Sections~\ref{AKSz:sec} and~\ref{Rodl:sec} we will discuss the range $\ell = O(\log k)$, in Section~\ref{closer:sec} we will sketch an elegant inductive version of the proof from~\cite{CGMS} for the range $\log k \ll \ell \ll k$, which was discovered by Gupta, Ndiaye, Norin and Wei~\cite{GNNW}, and in Section~\ref{diag:sec} we will outline the new (and much simpler) proof of Theorem~\ref{thm:diagonal} that was given in~\cite{BBCGHMST}. 

There has also been a recent breakthrough in the lower bound for $R(\ell,k)$ in this range, by Ma, Shen and Xie~\cite{MSX}, who used a random geometric graph to improve the bound given by a simple random colouring by an exponential factor. In Section~\ref{MSX:sec} we will describe their colouring, and provide a (very rough) heuristic explanation for why it works.  

\subsection{Induced Ramsey numbers}\label{induced:subsec}

The topic of the final section of this survey is a natural variant of the usual Ramsey numbers for \emph{induced} subgraphs. To define these numbers, let us write $G \xrightarrow{\mathrm{ind}} H$ if every red-blue colouring of the edges of $G$ contains an induced monochromatic copy of $H$ (that is, a copy of $H$ which is induced in $G$, and all the edges have the same colour). We then define
$$R^{\mathrm{ind}}(H) = \min \big\{ v(G): G \xrightarrow{\mathrm{ind}} H \big\}.$$
In particular, note that $R^{\mathrm{ind}}(K_k) = R(K_k)$. It is surprisingly challenging even to prove that these numbers are finite for every graph $H$, and the early proofs of this fact~\cite{D75,EHP,R73} gave bounds that were double-exponential or worse. Nevertheless, Erd\H{o}s~\cite{E75,E84} famously conjectured that $R^{\mathrm{ind}}(H)$ should be at most exponential in the number of vertices of $H$. This conjecture was recently proved by Aragão, Campos, Dahia, Filipe and Marciano~\cite{ACDFM}.

\begin{theorem}\label{thm:induced}
There exists a constant $C > 0$ such that
$$R^{\mathrm{ind}}(H) \le 2^{Ck}$$
for every graph $H$ with $k$ vertices.
\end{theorem}

We will outline the (extremely intricate) proof of Theorem~\ref{thm:induced} in Section~\ref{induced:sec}. The basic idea is to show that if $n \ge 2^{Ck}$, then the random graph $G(n,1/2)$ is a suitable choice for every graph $H$ with $k$ vertices; that is, we have
$$G(n,1/2) \xrightarrow{\mathrm{ind}} H$$
with (very) high probability. 
To do so, the authors reveal the edges of $G \sim G(n,1/2)$ inside a set $U$ of size $\delta n$, take a union bound over all choices of the colouring inside this set, and apply induction on $k$ to find a large and `well-distributed' collection of monochromatic induced copies of $H' = H - v$ inside $U$. Their main task is then to prove a suitably strong\footnote{Note that there are roughly $2^{|U|^2}$ choices for the colouring inside $U$, so their bound on the failure probability needs to be smaller than $2^{-\delta^2 n^2}$, which is not far from the trivial lower bound of $2^{-\delta n^2}$.} bound on the probability that there exists a colouring of the edges between $U$ and $V(G) \setminus U$ that does not extend these copies of $H'$ (which can now be considered to be fixed) to a large well-distributed collection of monochromatic induced copies of $H$. 

A key tool in this part of the proof is an exciting new variant of the method of hypergraph containers (see~\cite{BMS,ST}, or~\cite{ICM} for a gentle introduction to the method) which was discovered recently by Campos and Samotij~\cite{CS}. Roughly speaking, the authors show how this new tool can be used to reduce the study of `global' properties (such as being well-distributed) to `local' properties (which traditional container theorems are better-equipped to handle). It seems likely that this new method will have many further applications.

\subsection{Multicolour Ramsey numbers, and many other directions}

For simplicity, we have focused in this introduction on colourings with only two colours; in the sections below we will also discuss the more general setting of $r$-colourings, where many  beautiful problems remain open. We would also like to emphasize that in this survey we will only have space to discuss a few of the most recent advances in the area; for a much broader view of the development of graph Ramsey theory over the past few decades, and many further results and open problems, we recommend the excellent survey by Conlon, Fox and Sudakov~\cite{CFS}. 

The rest of this survey is organised as follows: in Sections~\ref{sec:R3k:upper} and~\ref{sec:R3k:lower} we will study the off-diagonal Ramsey numbers $R(3,k)$, and sketch the proof of Theorem~\ref{thm:R3k}; in Section~\ref{sec:R4k} we will sketch the proof of the Mattheus--Verstraete lower bound on $R(4,k)$; in Sections~\ref{AKSz:sec}--\ref{diag:sec} we will study bounds on $R(\ell,k)$ when $\ell \to \infty$, and outline the proof of Theorems~\ref{thm:diagonal} and~\ref{thm:off:diagonal}; and finally, in Section~\ref{induced:sec}, we will %discuss induced Ramsey numbers, and 
sketch the proof of Theorem~\ref{thm:induced}.

\section{Upper bounds on $R(3,k)$}\label{sec:R3k:upper}

We will begin fairly gently, by recalling a classical upper bound on $R(3,k)$, and some of the various known proofs. First, however, let us prove the Erd\H{o}s--Szekeres bound~\eqref{eq:ESz:bound}. 

\begin{theorem}[Erd\H{o}s and Szekeres, 1935]\label{thm:ESz:bound}
For every $\ell,k \in \N$, 
$$R(\ell,k) \le {k + \ell - 2 \choose \ell - 1}.$$ 
\end{theorem}

\begin{proof}
We claim that 
\begin{equation}\label{eq:ESz}
R(\ell,k) \le R(\ell-1,k) + R(\ell,k-1),
\end{equation}
from which the claimed bound follows easily by induction. To prove~\eqref{eq:ESz}, set $n = R(\ell,k) - 1$, and consider a red-blue colouring of $E(K_n)$ with no red copy of $K_\ell$ and no blue copy of $K_k$. Fix a vertex $v$, and observe that $v$ has at most $R(\ell-1,k) - 1$ red neighbours and at most $R(\ell,k-1) - 1$ blue neighbours, since otherwise we could add $v$ to complete a forbidden monochromatic clique. Counting vertices, we obtain~\eqref{eq:ESz}, as required.
\end{proof}

To improve this bound in the case $\ell = 3$, it will be useful to think of the problem in the following way. Let $G$ be the graph of red edges, so $G$ is triangle-free, and our aim is to find a large independent set in $G$ (which corresponds to a blue clique). Note that for every vertex $v \in V(G)$, the set $N(v)$ of neighbours of $v$ is an independent set, since $G$ is triangle-free. Thus the maximum degree of $G$ is at most $k - 1$, and hence\hspace{0.02cm}\footnote{As is standard in graph theory, we write $\alpha(G)$ for the size of the largest independent set in a graph $G$, and $\Delta(G)$ for the maximum degree of $G$. For background on graph theory, we refer the reader to~\cite{BGT}.}
\begin{equation}\label{eq:Turan}
\alpha(G) \ge \frac{n}{\Delta(G) + 1} \ge k
\end{equation}
if $n \ge k^2$. The first inequality can be proved via a greedy algorithm: in each step add an arbitrary vertex $v$ to our independent set, and remove $v$ and its neighbours from the set of available vertices. In the worst case we remove $\Delta(G) + 1$ vertices in each step. 

The basic idea of Ajtai, Komlós and Szemerédi's proof is that if we choose the vertices $v$ randomly, then the average degree of the graph on the remaining (available) vertices should go down, and hence for later choices the set should shrink by much less. Note that for this to be true we need some condition on the graph: for example, if $G$ were a union of cliques of size $\Delta(G) + 1$, then the bound~\eqref{eq:Turan} would be sharp. Perhaps surprisingly, it turns out that the assumption that $G$ is triangle-free suffices to avoid all such bad examples. 

A couple of years later, Shearer~\cite{Sh83} found a short and elegant argument that took the approach of~\cite{AKSz81} to its natural limit. In particular, he proved the following theorem.

\begin{theorem}[Shearer, 1983]\label{lem:Shearer}
Let $G$ be a triangle-free graph with $n$ vertices and average degree~$d$. Then 
$$\alpha(G) \ge \big(1 + o(1) \big) \frac{n \log d}{d}$$
as $d \to \infty$. 
\end{theorem}

\begin{proof}[Sketch of the proof] % of Theorem~\ref{lem:Shearer}] 
We will prove by induction on $n$ that $\alpha(G) \ge f(d) \cdot n$, where
$$f(d) = \frac{d \log d - d + 1}{(d-1)^2}.$$
Choose a random vertex $v$, and apply the induction hypothesis to the graph $G'$, obtained by deleting the vertices $\{v\} \cup N(v)$. It follows that
$$\alpha(G) \ge \Ex\big[ f(d') \big( n - d(v) - 1 \big) \big] + 1,$$ 
where $d'$ is the average degree of $G'$. The claim now follows from a short calculation, using the assumption that $G$ is triangle-free to show that
$$\Ex\big[ e(G') \big] = e(G) - \frac{1}{n} \sum_{v \in V(G)} d(v)^2,$$
and the following properties of the function $f$:
$$(d+1)f(d) = 1 + (d-d^2)f'(d) \qquad \text{and} \qquad f''(d) > 0$$
for all $d > 0$.
\end{proof}

The upper bound in Theorem~\ref{thm:R3k} follows almost immediately from Theorem~\ref{lem:Shearer}. 

\begin{proof}[Proof of the upper bound in Theorem~\ref{thm:R3k}]
Let $G$ be a triangle-free graph with $n$ vertices and no independent set of size $k$. Observe that $\Delta(G) < k$, since the neighbourhood of each vertex is an independent set. By Theorem~\ref{lem:Shearer}, it follows that 
$$k > \alpha(G) \ge \big(1 + o(1) \big) \frac{n \log k}{k}$$
as $k \to \infty$, and hence that 
$$n \le \big(1 + o(1) \big) \frac{k^2}{\log k},$$ 
as required.
\end{proof}

An important difference between the proof of Theorem~\ref{lem:Shearer} above, and the earlier proof (of a weaker bound) in~\cite{AKSz81}, is that in Shearer's proof we add one vertex at a time to the independent set, whereas Ajtai, Komlós and Szemerédi added roughly $n/d$ vertices in each step. A variant of this latter method, nowadays known as the `R\"odl nibble', was introduced by R\"odl~\cite{R85} in 1985 in order to prove a conjecture of Erd\H{o}s and Hanani~\cite{EH63} on the existence of approximate designs. This method has proved to be extremely powerful and flexible; for example, variants of it have been used in recent years to prove the following significant generalisations of Theorem~\ref{lem:Shearer}. For the first of these, let us write $\Delta_2(G)$ for the maximum co-degree (size of the common neighbourhood of two vertices) in $G$. 

\begin{theorem}[Campos, Jenssen, Michelen and Sahasrabudhe, 2023+]
Let $G$ be a graph with $n$ vertices, $\Delta(G) \le d$ and $\Delta_2(G) \le d / (\log d)^8$. Then 
$$\alpha(G) \ge \big(1 + o(1) \big) \frac{n \log d}{d}$$
as $d \to \infty$. 
\end{theorem}

This result was used by Campos, Jenssen, Michelen and Sahasrabudhe~\cite{CJMS23} to improve the best known lower bound on the density of a sphere packing in high dimensions. The following very recent result generalises Theorem~\ref{lem:Shearer} in a different direction. 

\begin{theorem}[Dhawan, Janzer and Methuku, 2025+]
Let $H$ be a graph with $\chi(H) = 3$, and let $G$ be an $H$-free graph with $n$ vertices and average degree~$d$. Then 
$$\alpha(G) \ge \big(1 + o(1) \big) \frac{n \log d}{d}$$
as $d \to \infty$. 
\end{theorem}

All of these proofs obtain the same bound because they find (very roughly speaking) a \emph{typical} independent set, and in a random $d$-regular graph most independent sets have this size. However, the \emph{largest} independent sets in such graphs are roughly twice as big, % as Shearer's bound
and it is a major open problem to determine which of these two bounds is closer to the truth. 
%The best known upper bound for Theorem~\ref{lem:Shearer} is given by the random $d$-regular graph, which has independence number roughly twice as large as Shearer's bound. 
In particular, note that any improvement of the bound in Theorem~\ref{lem:Shearer} would translate immediately into an improvement of the upper bound on $R(3,k)$. A positive answer to the following problem has been conjectured by many people over the years. 

%\begin{problem}\label{prob:Shearer:weak}
%Fix $\eps > 0$. Is it true that, if $d$ is sufficiently large, then
%$$\alpha(G) \ge ( 1 + \eps ) \frac{n \log d}{d}$$
%for every triangle-free graph $G$ with $n$ vertices and maximum degree~$d$?
%\end{problem}

\begin{problem}\label{prob:Shearer:strong}
Fix $\eps > 0$. Is it true that, if $d$ is sufficiently large, then
$$\alpha(G) \ge ( 2 - \eps ) \frac{n \log d}{d}$$
for every triangle-free graph $G$ with $n$ vertices and maximum degree~$d$?
\end{problem}

Indeed, even proving a weaker bound, with $2 - \eps$ replaced by $1 + \eps$, would be a major breakthrough. Similarly, it would be extremely interesting to find a counterexample.\footnote{Note that doing so would not necessarily improve the lower bound on $R(3,k)$; to do so, one would need a counterexample with $d \sim \sqrt{n \log n}$.}

Another piece of evidence in favour of a positive answer to Problem~\ref{prob:Shearer:strong} is the following theorem of Davies, Jenssen, Perkins and Roberts~\cite{DJPR}, which shows that if $G$ is a triangle-free graph with maximum degree at most $d$, then even the \emph{average} size of an independent set in $G$ is at least as large as the bound given by Shearer's theorem. 

\begin{theorem}[Davies, Jenssen, Perkins and Roberts, 2018]\label{thm:R3k:average}
Let $G$ be a triangle-free graph with $n$ vertices and $\Delta(G) \le d$, and let $S$ be a random independent set, chosen uniformly from all of the independent sets of $G$. Then 
$$\Ex\pigl[ |S| \pigr] \ge \big(1 + o(1) \big) \frac{n \log d}{d}$$
as $d \to \infty$. 
\end{theorem}

The proof of Theorem~\ref{thm:R3k:average} relies on a connection to the hard-core model from statistical physics. For simplicity, we will instead give a beautiful proof of a slightly weaker bound, which was discovered by Shearer~\cite{Sh95} in 1995. More precisely, we will follow the elegant presentation of Alon~\cite{A96}, who extended the proof to graphs that are `locally-sparse', in the sense that neighbourhoods induce subgraphs with bounded chromatic number.

\begin{proof}[Proof of Theorem~\ref{thm:R3k:average} up to a constant factor]
%The idea is to bound the \emph{average} size of an independent set in a triangle-free graph $G$ with maximum degree $d$. To do so, 
The key idea is to define, for each vertex $v \in V(G)$, a random variable 
$$X_v = \pigl| N(v) \cap S \pigr| + d \cdot \1\big[ v \in S \big].$$
%where $S$ is a random independent set, chosen uniformly from all of the independent sets of $G$. 
Observe that, by linearity of expectation, 
\begin{equation}\label{eq:Shearer:sumXv}
\sum_{v \in V(G)} \Ex\pigl[ X_v \pigr] \le 2d \cdot \Ex\pigl[ |S| \pigr],
\end{equation}
since each vertex has at most $d$ neighbours. Now fix a vertex $v \in V(G)$, and reveal the set $S$ outside the set $N^\circ(v) = \{v\} \cup N(v)$. We claim that 
\begin{equation}\label{eq:Shearer:via:Alon}
\Ex\big[ \hspace{0.05cm} X_v \;\big|\; S \setminus N^\circ(v) = T  \, \big] \ge \frac{\log_2 d}{6}
\end{equation}
for \emph{every} possible choice of $T$, and hence that the same lower bound holds for $\Ex\pigl[ X_v \pigr]$. 

To prove~\eqref{eq:Shearer:via:Alon}, observe that either $S = T \cup \{v\}$, or $S \setminus T$ is a subset of
$$Y = \big\{ u \in N(v) : N(u) \cap T = \emptyset \big\}.$$
Note also that $Y$ is an independent set, since $Y \subset N(v)$ and $G$ is triangle-free, and therefore each of these $2^{|Y|} + 1$ possibilities has the same probability, by the definition of $S$. By the definition of $X_v$, it follows that
$$\Ex\big[ \hspace{0.05cm} X_v \;\big|\; S \setminus N^\circ(v) = T \, \big] = \sum_{Z \subset Y} \frac{|Z|}{2^{|Y|} + 1} + \frac{d}{2^{|Y|} + 1} \ge \frac{\log_2 d}{6},$$
as required, since the sum is at least $|Y|/3$, and if $2|Y| < \log_2 d$ then the second term is large. Combining~\eqref{eq:Shearer:sumXv} and~\eqref{eq:Shearer:via:Alon} gives
$$\Ex\pigl[ |S| \pigr] \ge \frac{1}{2d} \sum_{v \in V(G)} \Ex\pigl[ X_v \pigr] \ge \frac{n \log_2 d}{12d},$$
as required. %so the average size of an independent set in $G$ is within a constant factor of the bound given by Theorem~\ref{lem:Shearer}. 
\end{proof}

The authors of~\cite{DJPR} %Davies, Jenssen, Perkins and Roberts 
moreover conjectured that the maximum size of an independent set in a triangle-free graph of minimum degree $d$ should be at least $2 - o(1)$ times the average size, as $d \to \infty$; similarly to Problem~\ref{prob:Shearer:strong}, any lower bound better than $1 + o(1)$ would constitute a very significant breakthrough. 

Finally, let us mention one more beautiful open problem. 

\begin{problem}\label{prob:Shearer:cliques}
Fix $\ell \ge 4$. Does there exist a constant $c = c(\ell) > 0$ such that
\begin{equation}\label{eq:Shearer:cliques}
\alpha(G) \geq \frac{cn \log d}{d}
\end{equation}
for every $K_\ell$-free graph $G$ with $n$ vertices and maximum degree~$d$?
\end{problem}

For $\ell \ge 4$, the best-known bound for $K_\ell$-free graphs was proved by Shearer~\cite{Sh95} in 1995, and falls short of~\eqref{eq:Shearer:cliques} by a factor of $\log\log d$.

% but no improvement of the trivial $1 + o(1)$ is known. 

%There exist graphs (for example, the complete graph) for which the maximum and average sizes of an independent set are asymptotically equal. However, no such triangle-free graphs are known, 

\section{Lower bounds on $R(3,k)$}\label{sec:R3k:lower}

In this section we will describe seven different constructions, proving successively stronger lower bounds on $R(3,k)$, culminating in the colouring of Hefty, Horn, King and Pfender~\cite{HHKP} which implies the lower bound in Theorem~\ref{thm:R3k}. 

\subsection{A geometric construction}\label{sec:Erdos:geometry}

The first non-trivial lower bound on $R(3,k)$ was proved by Erd\H{o}s~\cite{E57} in 1957, who used an explicit geometric construction to show that
$$R(3,k) \ge k^{1+c}$$
for some constant $c > 0$. We will describe a slight variant of Erd\H{o}s' colouring, which relies on the following elegant theorem of Kleitman~\cite{K66}. % which confirmed a conjecture of Erd\H{o}s.

\begin{theorem}[Kleitman, 1966]\label{thm:Kleitman}
Let $A \subset \{-1,1\}^n$. If 
%\begin{equation}%\label{eq:Kleitman}
$$|A| > \sum_{i = 0}^d {n \choose i},$$
%\end{equation} 
then there exist $x,y \in A$ with $\<x,y\> < n - 4d$. 
\end{theorem}

Now define a graph $G$ with vertex set $\{-1,1\}^n$ and edge set 
$$E(G) = \big\{ xy \,:\, \< x, y \> < -n/3 \big\}.$$
Observe that $G$ is triangle-free, and that, by Kleitman's theorem applied with $d = n/3$, 
$$\alpha(G) \le \sum_{i = 0}^{n/3} {n \choose i} \le n {n \choose n/3} < 2^{(1-c)n},$$
for some constant $c > 0$. Setting $k = 2^{(1-c)n}$, it follows that $G$ is a triangle-free graph with at least $k^{1+c}$ vertices and $\alpha(G) < k$, as required.

\subsection{The first probabilistic construction} 

Just two years later, Erd\H{o}s~\cite{E59} took another important step forward, by giving the first lower bound for $R(3,k)$ using a random graph. The idea is to choose $p = p(n)$ so that $G(n,p)$ typically has fewer than $n/2$ triangles, and then remove one vertex from each. This produces a triangle-free graph $G$ with 
\begin{equation}\label{eq:alpha:Gnp}
\alpha(G) \le \alpha\big( G(n,p) \big) \le \frac{2\log(pn)}{p},
\end{equation}
since removing vertices cannot increase the independence number. To have fewer than $n/2$ triangles we need to take $p \le n^{-2/3}$ (so that $p^3 n^3 \le n$), and we therefore obtain the bound 
$$R(3,k) \gtrsim \bigg( \frac{k}{\log k} \bigg)^{3/2}.$$
%for some constant $c > 0$. 

\subsection{A better idea: removing edges} 

Removing a vertex from each seems like a very inefficient way of destroying triangles, especially when there is a very natural (and much more efficient) alternative: simply remove one edge from each instead. This introduces a problem, however; removing edges can increase the size of the largest independent set. 

Controlling this increase in the independence number is not easy, but it has a significant payoff: if we only need the number of triangles in $G(n,p)$ to be smaller than the number of edges, then we can take $p \approx n^{-1/2}$ (so that $p^3 n^3 \approx pn^2$). If we can show that~\eqref{eq:alpha:Gnp} still holds up to a constant factor, then we would obtain a bound of the form  
\begin{equation}\label{eq:Erdos:R3k}
R(3,k) \gtrsim \bigg( \frac{k}{\log k} \bigg)^{2}.
\end{equation}
This is exactly what Erd\H{o}s~\cite{E61} achieved in 1961, in a paper that was far ahead of its time. Simpler proofs of~\eqref{eq:Erdos:R3k} %this bound 
were later discovered by Spencer~\cite{S77}, using the Lovász Local Lemma (see~\cite[Chapter~5]{AS}), %~\cite{LLL}, 
and by Krivelevich~\cite{K95}. For example, the Local Lemma implies that if $p \approx n^{-1/2}$, then with extremely small but (crucially) \emph{non-zero} probability, 
$$K_3 \not\subset G(n,p) \qquad \text{and} \qquad \alpha\big( G(n,p) \big) = O\big(\sqrt{n} \log n \big),$$ 
which implies Erd\H{o}s' bound on $R(3,k)$. % (see~\cite[Chapter~5]{AS} for the details). 
Krivelevich, on the other hand, removed a maximal collection of edge-disjoint triangles from $G(n,p)$, and then used large deviation inequalities (including a beautiful inequality of  Erd\H{o}s and Tetali~\cite{ET}) to bound the probability that in doing so we remove every edge from some set of size $k$. 

\subsection{Kim's nibble} 

The final factor of $\log k$ separating the upper and lower bounds on $R(3,k)$ was finally removed by Kim~\cite{Kim} in 1995. To do so, he used a R\"odl-nibble-like process to construct a triangle-free graph $G$ with\footnote{Here we write $d(G)$ for the average degree of a graph $G$. The graph $G$ (and all of the other graphs in this section) can also be taken to be `almost regular', meaning that $d(v) = \big( 1 + o(1) \big) d(G)$ for every $v \in V(G)$.} %density $p \approx \sqrt{\frac{\log n}{n}}$ and
%$$\alpha(G) \approx \alpha\big( G(n,p) \big) = O\big(\sqrt{n \log n} \big).$$
$$d(G) = \Theta\big( \sqrt{n \log n} \big) \qquad \text{and} \qquad \alpha(G) \approx \alpha\big( G(n,p) \big) = O\big(\sqrt{n \log n} \big).$$
%$$e(G) = \Theta\big( n^{3/2} \sqrt{\log n} \big) \qquad \text{and} \qquad \alpha(G) \approx \alpha\big( G(n,p) \big) = O\big(\sqrt{n \log n} \big).$$
%$$e(G) = p {n \choose 2} \approx n^{3/2} \sqrt{\log n} \qquad \text{and} \qquad \alpha(G) \approx \alpha\big( G(n,p) \big) = O\big(\sqrt{n \log n} \big).$$
That is, a random-like triangle-free graph with density larger by a factor of $\sqrt{\log n}$ than the construction of Erd\H{o}s, and without significantly larger independent sets than the Erd\H{o}s--Rényi random graph with the same density. Note that, together with the upper bound of Ajtai, Komlós and Szemerédi~\cite{AKSz}, this implies that
$$R(3,k) = \Theta\bigg( \frac{k^2}{\log k} \bigg).$$
In each step of Kim's nibble, he added each edge that does not create a triangle with the previously-chosen edges independently at random with probability $\eps n^{-1/2}$, and then destroyed any triangles created in the process using ideas from the proof of Krivelevich~\cite{K95}.

\subsection{The triangle-free process} 

Just as Shearer tightened the Ajtai--Komlós--Szemerédi bound using a one-vertex-at-a-time version of their nibble, it is natural to try to tighten Kim's bound by adding edges one at a time, with each chosen uniformly at random from those that do not create a triangle. This \emph{triangle-free process} was actually suggested several years earlier, by Bollobás and Erd\H{o}s, and motivated Kim's approach. It is surprisingly difficult to control, however, and the first results using it were only obtained in 1995 by Erd\H{o}s, Suen and Winkler~\cite{ESW}, who used it to give yet another proof of Erd\H{o}s' bound~\eqref{eq:Erdos:R3k}, and then in 2009 by Bohman~\cite{Boh}, who used it to reprove Kim's lower bound. Finally, the process was tracked to its asymptotic end by Fiz Pontiveros, Griffiths and Morris~\cite{FGM} and Bohman and Keevash~\cite{BK}, proving the existence of a graph $G$ with
$$d(G) = \bigg( \frac{1}{\sqrt{2}} + o(1) \bigg) \sqrt{n \log n} \qquad \text{and} \qquad \alpha(G) \le \big( \sqrt{2} + o(1) \big) \sqrt{n \log n},$$
%$$e(G) = \bigg( \frac{1}{2\sqrt{2}} + o(1) \bigg) n^{3/2} \sqrt{\log n} \qquad \text{and} \qquad \alpha(G) \le \big( \sqrt{2} + o(1) \big) \sqrt{n \log n},$$
which immediately implies that 
\begin{equation}\label{eq:TFP:R3k}
R(3,k) \ge \bigg( \frac{1}{4} + o(1) \bigg) \frac{k^2}{\log k}
\end{equation}
as $k \to \infty$. The proofs of this result in~\cite{BK} and~\cite{FGM} are \emph{extremely} complicated, involving the careful control of several large families of random variables that interact with one another in complex ways. Fortunately, as we will see below, there turns out to be a much simpler way to prove even stronger lower bounds on $R(3,k)$.

%Therefore, instead of discussing the proof, let us briefly give a heuristic for the final number of edges in the graph $G$ produced by the process. 

%To do so, let $e(G) = p {n \choose 2}$, and pretend that $G$ is a copy of $G(n,p)$ that happens (by luck) to contain no triangles. The density of `open' edges (that is, edges whose addition to $G$ would not create a triangle) is then roughly
%$$\pigl( 1 - p^2 \pigr)^{n-1} \approx e^{-p^2 n} \approx p \qquad \text{if} \qquad p \sim \sqrt{ \frac{\log n}{2n} },$$
%and so beyond this point we will not be able to increase $p$ significantly, even if we add in all remaining open edges. It turns out that this heuristic correctly predicts the final number of edges in a  typical graph produced by the process, and moreover that the independence number of $G$ is asymptotically equal to $\alpha\pigl( G(n,p) \pigr)$ for this value of $p$. 

%\subsection{Campos--Jenssen--Michelen--Sahasrabudhe add a blow-up of $G(n,p)$} 
\subsection{Starting with a blow-up of $G(n,p)$} 

The factor of $4$ separating~\eqref{eq:TFP:R3k} from Shearer's upper bound is really two factors of $2$: one coming from the lack of progress on Problem~\ref{prob:Shearer:strong}, and the other  from the fact that the graph $G$ given by the triangle-free process satisfies
$$\alpha(G) = \big( 2 + o(1) \big) d(G),$$
that is, the largest independent sets are twice as large as the neighbourhood of a vertex. This therefore leaves open the possibility that there could exist a denser triangle-free graph $G$ that still satisfies $\alpha(G) \sim \alpha\big( G(n,p) \big)$ (with $p$ equal to the density of $G$). However, for more than a decade no-one was able to construct such a graph, and the authors of~\cite{FGM} even conjectured %\footnote{For the record, I take full responsibility for this incorrect conjecture!} 
that no such graph exists.

\pagebreak

This barrier was finally overcome earlier this year, by Campos, Jenssen, Michelen and Sahasrabudhe~\cite{CJMS25}, who showed that adding a `seed step' to the triangle-free process can produce a denser but still highly random-like triangle-free graph. More precisely, they considered a blow-up of the random graph $G(n/s,p)$, with $s = (\log n)^2$ and $p = \sqrt{ \frac{\log n}{6n}}$, meaning that each vertex is replaced by an independent set of size $s$, and each edge is replaced by a complete bipartite graph. Note that $G(n/s,p)$ has fewer triangles than edges, and it is therefore not difficult to remove them without destroying the pseudorandom properties of the random graph. They then use an elegant variant of Kim's nibble to add random edges, producing a final triangle-free graph $G$ with 
$$d(G) = \bigg( \frac{\sqrt{2}}{\sqrt{3}} + o(1) \bigg) \sqrt{n \log n} \qquad \text{and} \qquad \alpha(G) \le \bigg( \frac{\sqrt{3}}{\sqrt{2}} + o(1) \bigg) \sqrt{n \log n},$$
%$$e(G) = \bigg( \frac{1}{\sqrt{6}} + o(1) \bigg) n^{3/2} \sqrt{\log n} \qquad \text{and} \qquad \alpha(G) \le \bigg( \frac{\sqrt{3}}{\sqrt{2}} + o(1) \bigg) \sqrt{n \log n},$$
which implies that 
\begin{equation}\label{eq:CJMS:R3k}
R(3,k) \ge \bigg( \frac{1}{3} + o(1) \bigg) \frac{k^2}{\log k}
\end{equation}
as $k \to \infty$. In order to control Kim's nibble until its asymptotic end, they added a new `regularization step' between each nibble step, which allowed them to \emph{dramatically} simplify the analysis of the process. The idea of adding such a regularization steps to a nibble process goes back to the work of Alon, Kim and Spencer~\cite{AKS97} in the 1990s, but the method has recently been rediscovered by several authors (see, e.g.,~\cite{CJMS23,GKLO,MPS}), and is quickly developing into a key part of the toolkit of probabilistic combinatorics.

\subsection{The Alon--R\"odl method}\label{AlonRodl:sec}

Before describing the construction that proves the lower bound in Theorem~\ref{thm:R3k}, we need to mention one more key technique for proving lower bounds on Ramsey numbers, which was introduced by Alon and R\"odl~\cite{AR} in 2005. To do so, we will take a slight detour into the world of multicolour off-diagonal Ramsey numbers. 

The Ramsey number $R(3,3,k)$ is the smallest $n \in \N$ such that every red-blue-green colouring of the edges of $K_n$ contains either a red triangle, a blue triangle, or a green copy of $K_k$. It follows from the approach of Erd\H{o}s and Szekeres that $R(3,3,k) \le k^3$, and using the method of Ajtai, Komlós and Szemerédi~\cite{AKSz} this can be improved to 
\begin{equation}\label{eq:R33k:upper:bound}
R(3,3,k) \le \frac{Ck^3}{(\log k)^2}
\end{equation}
for some constant $C > 0$. However, until the work of Alon and R\"odl, it wasn't known whether or not $R(3,3,k) \gg R(3,k)$. Their simple, beautiful, and surprisingly powerful idea is illustrated by the following lemma. 

\begin{lemma}[Alon and R\"odl, 2005] \label{lem:AR:33k}
If there exists a triangle-free graph $G$ with $n$ vertices and fewer than $\sqrt{{n \choose k}}$ independent sets of size $k$, then $R(3,3,k) > n$. 
\end{lemma}

\begin{proof}
To prove the lemma we simply take two \emph{random} copies $G_R$ and $G_B$ of the graph $G$ (that is, we take independent random permutations of the vertex set), and count the expected number of independent sets of size $k$ in their union. Note that a set is independent in $G_R \cup G_B$ if and only if it is independent in both $G_R$ and $G_B$. Therefore the expected number of independent $k$-sets in the graph $G_R \cup G_B$ is less than $1$, and hence there exists a pair of permutations such that $\alpha(G_R \cup G_B) < k$, as required.
\end{proof}

Applying this lemma to a blow-up of an explicit optimally-pseudorandom triangle-free graph, constructed by Alon~\cite{A94} in 1994, and using a clever counting argument to bound the number of independent sets of size $k$ (see Lemma~\ref{lem:AR:counting}), they obtained the bound
\begin{equation}\label{eq:R33k:lower:bound}
R(3,3,k) \ge \frac{ck^3}{(\log k)^{4}}
\end{equation}
for some constant $c > 0$. They also used their technique to prove tight lower bounds on many other multicolour off-diagonal Ramsey numbers. 

\subsection{The Hefty--Horn--King--Pfender construction: two random blow-ups} 

We are finally ready to describe the colouring which proves the lower bound in Theorem~\ref{thm:R3k}. This construction was discovered very recently by Hefty, Horn, King and Pfender~\cite{HHKP}, who were inspired by the proofs of~\eqref{eq:CJMS:R3k} and~\eqref{eq:R33k:lower:bound} to ask the following (in hindsight) very natural question: what if we replace the nibble phase by another random blow-up of $G(n/s,p)$? 

%\smallskip
%\pagebreak

To be slightly more precise, let us construct a graph\footnote{This is not exactly the same as the construction in~\cite{HHKP}, but it is simpler to understand and has very similar properties. Even more recent applications of similar constructions can be found in~\cite{CJMPS} and~\cite{KSSW}.} in the following way: 
\begin{itemize}
\item[1.] Let $H_1$ and $H_2$ be independent copies of $G(n/s,p)$, where $s = (\log n)^2$ and $p = \sqrt{ \frac{\log n}{4n}}$.\smallskip
\item[2.] Remove an edge from each triangle in $H_i$ to form a triangle-free graph $H_i'$. \smallskip  
\item[3.] Blow-up $H_1'$ and $H_2'$ to form triangle-free graphs $G_1$ and $G_2$ with $n$ vertices, choose a random bijection between their vertex sets, and consider their union $G_1 \cup G_2$.\smallskip   
\item[4.] Remove an edge from each triangle in $G_1 \cup G_2$ to form a triangle-free graph $G$.\smallskip   
\end{itemize}

Note that, due to the blowing-up, each of $G_1$ and $G_2$ has independent sets that are much larger than $k \approx \sqrt{n \log n}$. However, and crucially, there are \emph{very few} such independent sets, since they have many pairs of vertices in the same part of the blow-up. If we can show that there are fewer than $\sqrt{{n \choose k}}$ such sets, then we can use the idea of Alon and R\"odl to show that (with positive probability) none of these sets survive in $G_1 \cup G_2$. 

To see that there is some hope of this working, observe that the expected number of independent $k$-sets in $G(n,p)$ is this small as long as (roughly) $2k > \alpha(G(n,p))$, which is (just barely) the case for our choice of parameters. Therefore, if we ignore coincidences (pairs of vertices in the same part of the blow-up) and the edge-removal steps, then we would be in good shape. Moreover, neither coincidences nor edge-removal from $H_1$ and $H_2$ turn out to be significant problems, since $s$ is fairly small and $H_1$ and $H_2$ contain very few triangles. 

The main issue is therefore to deal with triangles in $G_1 \cup G_2$. Unfortunately there are likely to be many such triangles, roughly $p^3n^3 \approx pn^2 \log n$, which is too many for a naive edge-deletion argument to work. Fortunately, however, they come in batches: if an edge of $G_1$ forms a triangle with two edges of $G_2$, then removing it will destroy not only that triangle, but $s = (\log n)^2$ other triangles! Using this fact, the authors of~\cite{HHKP} are able to destroy the triangles of $G_1 \cup G_2$ without significantly increasing the independence number, and hence show that 
$$d(G) = \big( 1 + o(1) \big) \sqrt{n \log n} \qquad \text{and} \qquad \alpha(G) \le \big( 1 + o(1) \big) \sqrt{n \log n}.$$
%$$e(G) = \bigg( \frac{1}{2} + o(1) \bigg) n^{3/2} \sqrt{\log n} \qquad \text{and} \qquad \alpha(G) \le \big( 1 + o(1) \big) \sqrt{n \log n}.$$
This implies that 
\begin{equation}\label{eq:HHKP:R3k}
R(3,k) \ge \bigg( \frac{1}{2} + o(1) \bigg) \frac{k^2}{\log k}
\end{equation}
as $k \to \infty$, and therefore completes the (sketch) proof of Theorem~\ref{thm:R3k}. 

%Finally, to finish this section, let us restate the following conjecture from~\cite{CJMS25}.

%\begin{conjecture}\label{conj:R3k}
%$$R(3,k) = \bigg( \frac{1}{2} + o(1) \bigg) \frac{k^2}{\log k}$$ 
%as $k \to \infty$. 
%\end{conjecture}

\section{A lower bound on $R(4,k)$}\label{sec:R4k}

In this section we will outline the proof of the lower bound in Theorem~\ref{thm:R4k}. 
%which was discovered recently by Mattheus and Verstraete~\cite{MaV}. 
For the reader's convenience, we restate the bound here. 

\begin{theorem}[Mattheus and Verstraete, 2024]\label{thm:R4k:again}
There exists a constant $c > 0$ such that
$$R(4,k) \ge \frac{ck^3}{(\log k)^4}$$ 
for all sufficiently large $k \in \N$. 
\end{theorem}

As we mentioned in the introduction, the proof relies on the existence of a certain algebraic object, called the Hermitian unital. This object provides us with the graph that forms the starting point of the Mattheus--Verstraete construction (see~\cite[Proposition~2]{MaV}).

\begin{lemma}\label{lem:Hermitian:unital}
For every prime $q$, there exists a graph $H$ with $n = \Theta(q^4)$ vertices that has the following properties:
\begin{itemize}
\item[$(a)$] $H$ is $d$-regular for some $d = \Theta(q^3)$.\smallskip
\item[$(b)$] $E(H)$ is the union %of a family $\cA$ 
of\/ $\Theta(q^3)$ edge-disjoint cliques of size $\Theta(q^2)$. \smallskip
\item[$(c)$] Every copy of $K_4$ in $H$ intersects one of these cliques %in $\cA$ 
in at least three vertices.
\end{itemize}
\end{lemma}

\begin{proof}[Sketch of the proof] 
Let $q$ be a prime, and define $U$ (the Hermitian unital) to be the set of all 1-dimensional subspaces of $(\F_{q^2})^3$ that are spanned by a point $(x,y,z)$ satisfying
$$x^{q+1} + y^{q+1} + z^{q+1} = 0.$$
The vertices of $H$ are the lines in the projective plane $\text{PG}(2,q^2)$ that intersect $U$ in exactly $q+1$ points (there are exactly $q^4 - q^3 + q^2$ such lines), and two lines form an edge if they intersect in a point of $U$. Thus, for each element $u \in U$, we have a clique corresponding to the $q^2$ lines passing through $u$, and these cliques are edge-disjoint, since pairs of lines intersect in at most one point. Moreover, there are $q^3 + 1$ cliques, and each vertex of $H$ is contained in exactly $q+1$ of them. Property $(c)$ was proved by O'Nan~\cite{ONan} in 1972; for a short proof, see~\cite[Proposition~1]{MaV}.
\end{proof}

\pagebreak

Given this object, the first step is to destroy all of the copies of $K_4$. By property~$(c)$, this can be done simply by replacing each clique by a triangle-free graph; for example, a complete bipartite graph. Let $\cA$ be the family of edge-disjoint cliques of size $\Theta(q^2)$ given by Lemma~\ref{lem:Hermitian:unital} that partition the edge set of $H$, and let $H'$ be the (random) graph obtained from $H$ by replacing each clique of $\cA$ by a random complete bipartite graph.\footnote{This idea was apparently first introduced by Brown and R\"odl~\cite{BR91} in 1991, and has been applied or rediscovered several times by various different sets of authors, see, e.g.,~\cite{C17,DR11,GJ20,KMV}.}

We now have a $K_4$-free graph, but unfortunately $H'$ has large independent sets (for example, the parts of each complete bipartite graph, which have size $\Theta(n^{1/2})$). However, as in the previous section, what really matters is that there are \emph{not too many} large independent sets, and this will allow us to destroy them using a probabilistic argument. More precisely, we will consider a random subset $S \subset V(H')$ of the vertices, and show that the expected number of independent sets of size $k$ that are contained in $S$ is less than $1$. 

To make this argument work, we need a fairly good bound on the number of independent sets of size $k$ in $H'$. Mattheus and Verstraete proved the following bound with $C = 2^{30}$. 

\begin{lemma}\label{lem:R4k:counting:ksets}
%For each $k \in \N$, 
With high probability $H'$ has at most
$${q^4 \choose Cq \log q} {Cq^2 \choose k}$$ % \le \bigg( \frac{q}{(\log q)^2} \bigg)^k$$
independent sets of size $k$. 
\end{lemma}

Before sketching the proof of Lemma~\ref{lem:R4k:counting:ksets}, let us note that it easily allows us to complete the proof of Theorem~\ref{thm:R4k:again}. % (losing some log-factors for simplicity). 
To do so, set $p = q^{-1}$, and let $S$ be a $p$-random subset of $V(H')$, meaning that each vertex is included in $S$ independently at random with probability~$p$. Then $G = H'[S]$ is a $K_4$-free graph with roughly $pn = \Theta(q^3)$ vertices, and the expected number of independent sets of size $k = q \, (\log q)^3$ in $G$ is at most
$$p^k {q^4 \choose Cq \log q} {Cq^2 \choose k} \le \bigg( \frac{pq}{(\log q)^2} \bigg)^k \to 0$$
as $q \to \infty$. It follows that there exists a graph $G$ with 
$$v(G) \ge \frac{ck^3}{(\log k)^9} \qquad \text{and} \qquad \alpha(G) < k$$
for some constant $c > 0$. With a little more care, this argument can easily be tightened to give the lower bound on $R(4,k)$ stated in Theorem~\ref{thm:R4k:again}. 

Mattheus and Verstraete prove Lemma~\ref{lem:R4k:counting:ksets} using the method of graph containers, which was introduced in 1982 by Kleitman and Winston~\cite{KW} in order to count the number of $C_4$-free graphs on $n$ vertices, and later developed much further by Sapozhenko~\cite{Sap}, see the survey by Samotij~\cite{Sam}. Very roughly, the container method says that the independent sets in graphs are (typically) clustered together, and can be covered by a relatively small number of sparse sets. More precisely, we have the following lemma (see~\cite[Lemma~1]{Sam}).

\pagebreak

\begin{lemma}\label{lem:graph:containers}
Let $G$ be a graph. If $\,\beta > 0$ and $R,s,k \in \N$ with $s \le k$ are such that
%Let $G$ be a graph, and let $k \in \N$. If there exist $R,s \in \N$ and $\beta > 0$ such that
\begin{equation}\label{eq:graph:containers:conditions}
R \ge e^{-\beta s} n \qquad \text{and} \qquad e(G[U]) \ge \beta |U|^2
\end{equation}
%$$R \ge e^{-\beta s} n \qquad \text{and} \qquad e(G[U]) \ge \beta {|U| \choose 2}$$
for every set $U \subset V(G)$ with $|U| \ge R$, then $G$ has at most
$${n \choose s} {R \choose k - s}$$
independent sets of size $k$. % for every $k \in \N$ with $k \ge s$. 
\end{lemma}

The proof of Lemma~\ref{lem:graph:containers} is roughly as follows: for each independent set $I$ of $G$, we find a `fingerprint' $S \subset I$ of size $s$, and a corresponding `container' $g(S) \supset I$ of size at most $R$, which only depends on $S$, not on the remainder of $I$. The claimed bound on the number of independent $k$-sets then follows immediately (we have at most ${n \choose s}$ choices for the fingerprint, and choose the remaining elements from the container).

Constructing the fingerprint and container is also surprisingly easy: we repeatedly choose a vertex of maximum degree in (the current candidate for) the container, remove it from the container if it is not in $I$, and otherwise add it to the fingerprint and remove its neighbours from the container. It follows from the assumptions~\eqref{eq:graph:containers:conditions} that after adding $s$ elements to the fingerprint, the container will have size at most $R$, as claimed. 

In order to apply Lemma~\ref{lem:graph:containers}, we need to prove a `supersaturation lemma' for $H'$; that is, we need to show that (with high probability) every large subset of $V(H')$ contains many edges of  $H'$. Mattheus and Verstraete proved the following lemma of this type.

\begin{lemma}\label{lem:R4k:supersaturation}
With high probability, the random graph $H'$ has the following property: 
$$e(G[U]) \ge \frac{|U|^2}{2^{10} q}$$ 
for every set $U \subset V(H')$ of size at least $Cq^2$. 
\end{lemma}

The proof of Lemma~\ref{lem:R4k:supersaturation} uses the properties of the graph $H$ guaranteed by Lemma~\ref{lem:Hermitian:unital}, together with a relatively straightforward martingale argument. The lemma allows us to apply Lemma~\ref{lem:graph:containers} with $R = Cq^2$ and $s = 2^{12} q\log q$, and immediately obtain Lemma~\ref{lem:R4k:counting:ksets}. As explained after Lemma~\ref{lem:R4k:counting:ksets}, this therefore completes the (sketch) proof of Theorem~\ref{thm:R4k:again}. 

\subsection{Optimally pseudorandom graphs and Ramsey numbers}\label{sec:MuV}

When $\ell \ge 5$ we have no analogue of the Hermitian unital to help us, and the best-known lower bounds are given by natural generalisations of the constructions described in Section~\ref{sec:R3k:lower}. These bounds differ from the Erd\H{o}s--Szekeres bound~\eqref{eq:ESz:bound} by a polynomial factor when $\ell \ge 5$, leaving us with the following rather unsatisfactory situation:
\begin{equation}\label{eq:Rlk:bounds}
k^{(\ell+1)/2 + o(1)} \le R(\ell,k) \le k^{\ell-1+o(1)}
\end{equation}
as $k \to \infty$. In particular, in the case $\ell = 5$ we have the following bounds: 
\begin{equation}\label{eq:R5k:bounds}
\frac{ck^3}{(\log k)^{8/3}} \le R(5,k) \le \frac{Ck^4}{(\log k)^3}
\end{equation}
for some constants $C,c > 0$. The lower bound follows from the analysis of the $H$-free process by Bohman and Keevash~\cite{BK10} (here we only need the case $H = K_5$), while the upper bound %\footnote{The constant $C$ was improved to $1+o(1)$ by Li, Rousseau and Zang~\cite{LRZ}, by combining the methods of Ajtai--Komlós--Szemerédi~\cite{AKSz} and Shearer~\cite{Sh83}.} 
was proved by Ajtai, Komlós and Szemerédi~\cite{AKSz}, see Section~\ref{AKSz:sec}.

One potentially promising approach towards improving the lower bound in~\eqref{eq:Rlk:bounds} was introduced fairly recently by Mubayi and Verstraete~\cite{MuV}, and was one of the motivations for the proof in~\cite{MaV}. In order to describe it, we need to define more precisely what we mean by a \emph{pseudorandom graph}; the following definition was introduced by Thomason~\cite{T87} in 1987.

\begin{defn}\label{def:jumbled}%[$(p,\beta)$-jumbled graphs]
A graph $G$ is \emph{$(p,\beta)$-jumbled} if
$$\bigg| e\big( G[U] \big) - p {|U| \choose 2} \bigg| \le \beta |U|$$
for every set $U \subset V(G)$. 
\end{defn}

It follows easily from Chernoff's inequality that $G(n,p)$ is $(p,\beta)$-jumbled with high probability for some $\beta = O(\sqrt{pn})$, and it was shown by Erd\H{o}s, Goldberg, Pach and Spencer~\cite{EGPS} that no graph on $n$ vertices is $(p,\beta)$-jumbled with $\beta = o(\sqrt{pn})$. We therefore say that a graph is \emph{optimally pseudorandom} if it is $(p,\beta)$-jumbled for some $p$ and $\beta = O\pigl( \sqrt{pn} \pigr)$. The following question is one of the  most important open problems in graph theory.

\begin{question}\label{qu:optimally:pseudorandom}
How dense can an optimally pseudorandom $K_\ell$-free graph be?
\end{question} 

It is not hard to show that if $\beta = o\pigl( p^{\ell-1} n \pigr)$ then every $(p,\beta)$-jumbled graph with $n$ vertices contains a copy of $K_\ell$ (just apply the definition to the neighbourhood of a vertex, and use induction on $\ell$), so a necessary condition is that
\begin{equation}\label{eq:pseudorandom:upper}
p = O\big( n^{-1/(2\ell-3)} \big).
\end{equation}
%so that $p^{\ell-1} n \leqslant \beta \leqslant \sqrt{pn}$.
In the case $\ell = 3$ such graphs exist: an optimally pseudorandom triangle-free graph with density $n^{-1/3}$ was discovered by Alon~\cite{A94} in 1994, and another (random) construction was given by Conlon~\cite{C17}.\footnote{In~\cite{C17} it is only shown that there exist $(p,\beta)$-jumbled graphs with $p = \Theta(n^{-1/3})$ and $\beta = O\big( \sqrt{pn} \log n \big)$; however, as noted in~\cite{MaV}, the proof of Lemma~\ref{lem:R4k:supersaturation} allows one to remove the factor of $\log n$.} 
However, for $\ell \ge 4$ the best known constructions of optimally pseudorandom $K_\ell$-free graphs have density 
\begin{equation}\label{eq:pseudorandom:lower}
p = \Theta\big( n^{-1/(\ell-1)} \big).
\end{equation}
These graphs were discovered in 2020 by Bishnoi, Ihringer and Pepe~\cite{BIP}, who improved an earlier construction of Alon and Krivelevich~\cite{AK97}. The vertices of their graphs are the set of square points in the $(\ell-1)$-dimensional projective space $\text{PG}(\ell-1,q)$ over a finite field $\F_q$, and the edges correspond to the zeros of a certain quadratic form. 

Despite the large gap between~\eqref{eq:pseudorandom:upper} and~\eqref{eq:pseudorandom:lower}, it is widely believed that there do exist optimally pseudorandom $K_\ell$-free graphs with density $n^{-1/(2\ell-3)}$. Mubayi and Verstraete~\cite{MuV} showed that \emph{if} such graphs exist, then the upper bound in~\eqref{eq:Rlk:bounds} is also tight. 

\begin{theorem}[Mubayi and Verstraete, 2024]\label{thm:MuV}
%If there exists a $(p,\beta)$-jumbled $K_\ell$-free graph with $n$ vertices and $\beta = O(p^{\ell-1} n)$, then 
If there exists an optimally pseudorandom $K_\ell$-free graph with $n$ vertices and density $p = \Theta\pigl( n^{-1/(2\ell-3)} \pigr)$, then 
$$R(\ell,k) \ge \frac{ck^{\ell-1}}{(\log k)^{2\ell - 4}}$$
for some constant $c > 0$. 
\end{theorem} 

In fact, the same conclusion holds under the slightly weaker assumption that there exists a $(p,\beta)$-jumbled $K_\ell$-free graph with $n$ vertices and $\beta = \Theta\pigl( p^{\ell-1} n \pigr)$. 

The proof of Theorem~\ref{thm:MuV} is surprisingly simple: one just needs to consider a $q$-random subset $S$ of the vertices for some suitable function $q = q(n)$, %$p \approx \frac{(\log n)^2}{\beta}$. 
and bound the expected number of independent $k$-sets in $S$. To do so, we will use the following bound of Alon and R\"odl~\cite{AR}, which was already mentioned in Section~\ref{AlonRodl:sec}. Their lemma gives a general upper bound on the number of independent $k$-sets in a $(p,\beta)$-jumbled graph. 

\begin{lemma}[Alon and R\"odl, 2005]\label{lem:AR:counting}
Let\/ $G$ be a $(p,\beta)$-jumbled graph with $n$ vertices. If\/ $k \ge \frac{(\log n)^2}{p}$, 
then $G$ has at most 
$$\bigg( \frac{2^{10}\beta}{pk} \bigg)^k$$
independent sets of size $k$. 
\end{lemma}

To prove Lemma~\ref{lem:AR:counting}, choose the vertices of the independent set one by one, as usual removing the neighbourhoods of the selected vertices from the set $A$ of available vertices. Now observe that $A$ can shrink by a factor of $1 - p/2$ in at most $O\big( \frac{\log n}{p} \big)$ steps, and use the assumption that $G$ is $(p,\beta)$-jumbled to bound the number of choices in all remaining steps. 

Now, to prove Theorem~\ref{thm:MuV}, let $G$ be a $(p,\beta)$-jumbled $K_\ell$-free graph with $n$ vertices, set 
$$q = \frac{(\log n)^2}{\beta} \qquad \text{and} \qquad k = \frac{2^{11}(\log n)^2}{p},$$ 
and let $S$ be a $q$-random subset of $V(G)$. Then $|S| \approx qn$ and the expected number of independent sets of size $k$ in $S$ is at most
$$q^k \bigg( \frac{2^{10}\beta}{pk} \bigg)^k \to 0$$
as $k \to \infty$, by Lemma~\ref{lem:AR:counting}. We therefore obtain a $K_\ell$-free graph $G[S]$ with roughly $qn$ vertices and no independent set of size $k$. Moreover, if $\beta = \Theta\pigl( p^{\ell-1} n \pigr)$, then %$p = \Theta\big( n^{-1/(2\ell-3)} \big)$ and $\beta = \Theta(\sqrt{pn})$, then
$$|S| = \Theta\bigg( \frac{n(\log n)^2}{\beta} \bigg) = \Theta\bigg( \frac{k^{\ell-1}}{(\log k)^{2\ell - 4}} \bigg),$$
as required.

%note that applying this proof to the Bishnoi, Ihringer and Pepe graph gives
%$$|S| = \Theta\bigg( \frac{n(\log n)^2}{\beta} \bigg) = \Theta\bigg( \frac{k^{\ell/2}}{(\log k)^{2\ell - 4}} \bigg),$$
%\beta = p^s n if  pn = p^{2s} n^2 if n^{-1} = p^{2s-1} so 2s - 1 = ell - 1
%p = \Theta\big( n^{-1/(\ell-1)} \big).

\section{The Ajtai--Komlós--Szemerédi method}\label{AKSz:sec} 

In this section we will prove the following theorem of Ajtai, Komlós and Szemerédi~\cite{AKSz}, which gives (essentially) 
%\footnote{For fixed $\ell$ the constant factor was improved %to $1+o(1)$ 
%by Li, Rousseau and Zang~\cite{LRZ}.} %, by combining the methods of Ajtai--Komlós--Szemerédi~\cite{AKSz} and Shearer~\cite{Sh83}.} 
the best-known upper bound on $R(\ell,k)$ for all $3 \le \ell \ll \log k$. The~theorem is only stated in~\cite{AKSz} for fixed $\ell$ and $k \to \infty$, but the full version follows easily from the same proof. Since their method is both simple and beautiful, and moreover is not as widely-known as it should be, we provide an essentially complete proof. 

\begin{theorem}[Ajtai, Komlós and Szemerédi, 1980]\label{thm:AKSz}
Let $k \in \N$ be sufficiently large. Then 
$$R(\ell,k) \le \bigg( \frac{8\ell}{\log k} \bigg)^{\ell-2} {k + \ell - 2 \choose \ell - 1}.$$
for every $\ell \ge 2$. 
\end{theorem}

Note in particular that Theorem~\ref{thm:AKSz} implies~\eqref{eq:AKSz:fixed:ell} for all fixed $\ell$, and Theorem~\ref{thm:off:diagonal} for all $3 \le \ell \le (\log k)/9$. The first step %(cf.~\cite[Lemma~5]{AKSz})
is to deduce the following bound on the independence number of graphs with few triangles from Theorem~\ref{lem:Shearer}.

\begin{lemma}\label{cor:few:triangles}
Let $G$ be a graph with $n$ vertices, average degree at most $d$, and at most $d^2 n / \lambda^3$ triangles for some $\lambda = \lambda(d)$ with $1 \ll \lambda \le d$. Then 
$$\alpha(G) \ge \big(1 + o(1) \big) \frac{n \log \lambda}{d}$$
as $d \to \infty$. 
\end{lemma}

\begin{proof}
Set $p = \lambda/d$, and let $S$ be a $p$-random subset of $V(G)$. The expected number of triangles in $G[S]$ is at most $p^3 d^2 n / \lambda^3 = n / d$, and therefore, by Markov's inequality, with probability at least $1/2$ the subgraph $G[S]$ induced by $S$ contains at most $2n/d$ triangles. Moreover, by Chernoff's inequality, with probability at least $2/3$ we have 
$$|S| = \big(1 + o(1) \big) \frac{\lambda n}{d} \qquad \text{and} \qquad e(G[S]) \le \big(1 + o(1) \big) \frac{\lambda^2 n}{2d}.$$ 
Therefore, removing one vertex from each triangle in $G[S]$, we obtain a triangle-free induced subgraph $G' \subset G$ with $\big(1 + o(1) \big) \lambda n /d$ vertices and average degree at most $\big(1 + o(1) \big) \lambda$. 

Applying Theorem~\ref{lem:Shearer} to this graph, we deduce that 
$$\alpha(G) \ge \alpha(G') \ge \big(1 + o(1) \big) \frac{(\lambda n / d) \log \lambda}{\lambda} = \big(1 + o(1) \big) \frac{n \log \lambda}{d},$$
as claimed.
\end{proof}

We can now bound $R(\ell,k)$ by induction on $\ell$. %(cf.~\cite[Theorem~6]{AKSz}).

\begin{proof}[Proof of Theorem~\ref{thm:AKSz}]
Let $k \in \N$ be sufficiently large. We will prove by induction on $\ell$ that 
\begin{equation}\label{eq:AKSz:bound}
R(\ell,k) \le \bigg( \frac{8}{\log k} \bigg)^{\ell-2} k^{\ell-1}
\end{equation}
for every $\ell \ge 2$, which easily implies the claimed bound. Note that~\eqref{eq:AKSz:bound} holds when $\ell = 2$, since $R(2,k) = k$, and recall that we proved the case $\ell = 3$ in Section~\ref{sec:R3k:upper}. 

Now let $\ell \ge 4$, and assume that~\eqref{eq:AKSz:bound} holds for $\ell - 1$ and $\ell - 2$. Set $n = R(\ell,k) - 1$, and let $G$ be a $K_\ell$-free graph with $n$ vertices and no independent set of size $k$. Set 
$$d = \bigg( \frac{8}{\log k} \bigg)^{\ell-3} k^{\ell-2}$$
and observe that the maximum degree of $G$ is at most $d$, since every vertex of $G$ has degree less than $R(\ell-1,k)$, and by the induction hypothesis we have $R(\ell-1,k) \le d$. 

Set $\lambda = (8k / \log k)^{1/3}$ and suppose that some vertex $v$ is contained in at least $d^2 / \lambda^3$ triangles in $G$. Then the neighbourhood $N(v)$ induces a graph $G'$ with  
$$v(G') \le d \qquad \text{and} \qquad e(G') \ge \frac{d^2}{\lambda^3} = \bigg( \frac{8}{\log k} \bigg)^{\ell-4} k^{\ell-3} \cdot d.$$
By the induction hypothesis, it follows that
$$\Delta(G') \ge \bigg( \frac{8}{\log k} \bigg)^{\ell-4} k^{\ell-3} \ge R(\ell-2,k),$$
which is a contradiction, since $G$ is $K_\ell$-free and $\alpha(G) < k$.

It follows that there are at most $d^2 n / \lambda^3$ triangles in $G$. Applying Lemma~\ref{cor:few:triangles} to $G$, we deduce that 
$$k > \alpha(G) \ge \big(1 + o(1) \big) \frac{n \log \lambda}{d},$$ 
and therefore 
\begin{equation}\label{eq:ell:vs:ell:minus:one:better}
R(\ell,k) = n + 1 \le \frac{2kd}{\log \lambda} \le \bigg( \frac{8}{\log k} \bigg)^{\ell-2} k^{\ell-1},
\end{equation}
as required, since $k$ is sufficiently large and $\lambda \ge k^{1/4}$. 
\end{proof}

\section{R\"odl's method:~counting Erd\H{o}s--Szekeres paths}\label{Rodl:sec}

In this section we will present an approach due to R\"odl (see~\cite[Theorem~2.13]{GR}), which we will use to deduce Theorem~\ref{thm:off:diagonal} in the range $\ell = \Theta( \log k )$ from Theorem~\ref{thm:AKSz}. 

%The case $\ell = \log k$ of Corollary~\ref{cor:Rodl} is stated as Theorem~2.13 in~\cite{GR}

\begin{theorem}[R\"odl, 1987]\label{cor:Rodl}
%There exists $c > 0$ such that
If $\hspace{0.04cm} c > 0$ is sufficiently small, then
$$R(\ell,k) \le k^{-c} \hspace{0.02cm} {k + \ell - 2 \choose \ell-1}.$$
for all sufficiently large $\ell,k \in \N$ with $c \log k \le \ell \le c\sqrt{k}$. 
\end{theorem}

Since R\"odl's method seems to be even less well known than that of Ajtai, Komlós and Szemerédi, we will give the details. The idea is to apply the Erd\H{o}s--Szekeres inequality %~\eqref{eq:ESz} 
\begin{equation}\label{ineq:ESz}
R(\ell,k) \le R(\ell-1,k) + R(\ell,k-1)
\end{equation}
repeatedly, stopping when we reach a pair $(\ell',k')$ with $\ell' \le c \log k'$. We then apply the Ajtai--Komlós--Szemerédi bound, Theorem~\ref{thm:AKSz}, winning a small polynomial factor over the Erd\H{o}s--Szekeres bound as long as $k'$ is not too small. To complete the proof, we will bound the probability that a random Erd\H{o}s--Szekeres path does not pass through a pair $(\ell',k')$ with $\ell' \le c \log k'$ until $k'$ is small. To be precise, for each set $L \in {[k+\ell-2] \choose \ell-1}$, define 
$$m(L) = \max\big\{ 1 \le m \le k+\ell-2 \,:\, |L \cap [m]| \le c \log m \,\text{ or }\, [m] \subset L \big\}$$ 
%$$m(L) = \max\big\{ 1 \le m \le k+\ell-2 \,:\, |L \cap [m]| \le c \log m \,\text{ or }\, |L \cap [m]| \in \{1,m - 1\} \big\}$$ 
%to be the first step at which either $\ell' \le c \log (k' + \ell')$ or $k' = 1$, and
and set
$$\ell'(L) = |L \cap [m(L)]|+1 \qquad \text{and} \qquad k'(L) = m(L) - \ell'(L) + 2.$$ 
The following inequality will allow us to bound $R(\ell,k)$ using Theorem~\ref{thm:AKSz}. 

\begin{lemma}\label{lem:running:ESz}
For every $\ell,k \in \N$, we have 
\begin{equation}\label{eq:Rodl:sum}
R(\ell,k) \le \sum_{L \hspace{0.02cm} \in \hspace{0.02cm} {[k+\ell-2] \choose \ell-1}} {k'(L) + \ell'(L) - 2 \choose \ell'(L)-1}^{-1} R\big( \ell'(L), k'(L) \big).
\end{equation}
\end{lemma}

\begin{proof}
The proof is by induction; note that it holds trivially if either $\ell - 1 \le c \log(k+\ell-2)$ or $k = 1$, since then $\ell'(L) = \ell$ and $k'(L) = k$ for every $L \in {[k+\ell-2] \choose \ell-1}$. We may therefore assume that $\ell > c \log(k+\ell-2) + 1$, that $k \ge 2$, and that the inequality is true for the pairs $(\ell-1,k)$ and $(\ell,k-1)$. By~\eqref{ineq:ESz} and the induction hypothesis, it follows that
$$R(\ell,k) \, \le \sum_{L \hspace{0.02cm} \in \hspace{0.02cm} {[k+\ell-3] \choose \ell-1} \hspace{0.02cm} \cup \hspace{0.02cm} {[k+\ell-3] \choose \ell-2}} {k'(L) + \ell'(L) - 2 \choose \ell'(L) - 1}^{-1} R\big( \ell'(L), k'(L) \big).$$
Now, since $\ell > c \log(k+\ell-2) + 1$ and $k \ge 2$, it follows that if a set $L' \in {[k+\ell-2] \choose \ell-1}$ is either equal to $L \in {[k+\ell-3] \choose \ell-1}$, or is obtained from $L \in {[k+\ell-3] \choose \ell-2}$ by adding the element $k+\ell-2$, then $\ell'(L) = \ell'(L')$ and $k'(L) = k'(L')$, and hence this is exactly the claimed inequality. 
\end{proof}

Alternatively, note that there are ${k' + \ell'-2 \choose \ell'-1}$ (or zero) sets $L$ with $\ell'(L) = \ell'$ and $k'(L) = k'$ that have a given intersection with the set $\{k'+\ell'-1,\ldots,k+\ell-2\}$.

Now, for each $k,\ell \in \N$, define 
$$\cL(k,\ell) = \bigg\{ L \in {[k+\ell-2] \choose \ell-1} \,:\, k'(L) \ge \sqrt{k} \, \bigg\}.$$
The following simple lemma shows that almost all sets in ${[k+\ell-2] \choose \ell-1}$ are also in $\cL(k,\ell)$. 

\begin{lemma}\label{lem:Rodl:few:bad:seqs}
If\/ $k \in \N$ is sufficiently large and $2 \le \ell \le c\sqrt{k}$, then 
$$\bigg| {[k+\ell-2] \choose \ell-1} \setminus \cL(k,\ell) \bigg| \le k^{-2c} {k + \ell - 2 \choose \ell-1}.$$
\end{lemma}

\begin{proof}
If $\ell'(L) + k'(L) < t$, then $|L \cap [t]| > c \log t$. The number of sets $L \in {[k+\ell-2] \choose \ell-1}$ for which this is true is at most
$${t \choose c \log t} {k+\ell-2 \choose \ell - 1 - c \log t} \le \bigg( \frac{\ell \cdot t}{k} \bigg)^{c\log t} {k + \ell - 2 \choose \ell - 1}.$$
Applying this with $t = 2\sqrt{k}$ gives the claimed bound. 
\end{proof}

We can now easily deduce R\"odl's theorem.

\begin{proof}[Proof of Theorem~\ref{cor:Rodl}]
By Lemma~\ref{lem:running:ESz}, it will suffice to bound the right-hand side of~\eqref{eq:Rodl:sum}. When $L \not\in \cL(k,\ell)$, we do so using the usual Erd\H{o}s--Szekeres bound~\eqref{eq:ESz:bound}, which allows us to bound each summand by $1$. On the other hand, if $L \in \cL(k,\ell)$ and $\ell \le c\sqrt{k}$, then 
$$\ell'(L) = c \log k'(L) + O(1).$$ 
Therefore, if $k$ is sufficiently large, then by Theorem~\ref{thm:AKSz} we have
$${k'(L) + \ell'(L) - 2 \choose \ell'(L) - 1}^{-1} R\big( \ell'(L), k'(L) \big) \le \bigg( \frac{8\ell'(L)}{\log k'(L)} \bigg)^{\ell'(L)-2} \le \hspace{0.02cm} k^{-2c}.$$
Hence, by Lemmas~\ref{lem:running:ESz} and~\ref{lem:Rodl:few:bad:seqs}, %and using the usual Erd\H{o}s--Szekeres bound~\eqref{eq:ESz:bound} for those $L \not\in \cL(k,\ell)$, 
we deduce that
$$R(\ell,k) \le 2 \cdot k^{-2c} \hspace{0.02cm} {k + \ell - 2 \choose \ell - 1},$$
as required. 
\end{proof}

\section{Ramsey numbers closer to the diagonal}\label{closer:sec}

In this section we will sketch the proof of the following theorem of Gupta, Ndiaye, Norin and Wei~\cite{GNNW}, which they obtained using a streamlined and optimised version of the method of Campos, Griffiths, Morris and Sahasrabudhe~\cite{CGMS}. 

\begin{theorem}[Gupta, Ndiaye, Norin and Wei, 2024+]\label{thm:GNNW}
There exists $C > 0$ such that  
\begin{equation}\label{eq:GNNW:thm}
R(\ell,k) \le k^C \bigg( \frac{\sqrt{5} + 1}{4} \bigg)^\ell {k + \ell \choose \ell}
\end{equation}
for every $\ell,k \in \N$ with $\ell \ll k$. 
\end{theorem}

Note that this implies Theorem~\ref{thm:off:diagonal} for all $\log k \ll \ell \ll k$. The proof of Theorem~\ref{thm:GNNW} given in~\cite{GNNW} is quite short, but not very transparent, and requires some careful calculation, which we would rather avoid. We will therefore restrict ourselves to describing the main ideas, and refer the reader to~\cite[Section~2]{GNNW} for the details. 

To warm ourselves up for the proof, let us first consider the following slightly weaker version of the Erd\H{o}s--Szekeres bound~\eqref{eq:ESz:bound}:
\begin{equation}\label{eq:ESz:weaker}
R(\ell,k) \le \bigg( \frac{k+\ell}{\ell} \bigg)^\ell \bigg( \frac{k+\ell}{k} \bigg)^k.
\end{equation}
To prove~\eqref{eq:ESz:weaker}, we will build a red clique $A$ and a blue clique $B$ by adding one vertex at a time to one of the two cliques. To be more precise, suppose we have three sets $A$, $B$ and $X$, and that all edges inside $A$ and between $A$ and $X$ are red, and all edges inside $B$ and between $B$ and $X$ are blue. % (see Figure~\ref{fig:ESz}). 
Choose any vertex $x \in X$, add $x$ to $A$ if 
$$\pigl| N_R(x) \cap X \pigr| \ge \bigg( \frac{\ell}{k+\ell} \bigg) |X|,$$
and otherwise add $x$ to $B$. Moreover, replace $X$ by either $N_R(x)$ or $N_B(x)$, so that the edges between the sets are still all the same colour. If $n$ is at least the right-hand side of~\eqref{eq:ESz:weaker}, then we can continue until either $A$ has size $\ell$, or $B$ has size $k$, as required.
 
To improve the bound~\eqref{eq:ESz:weaker}, we will introduce a new set $Y$, which is contained in the common blue neighbourhood of the vertices in $B$ (see Figure~\ref{fig:book:alg}), and attempt to control the density of blue edges between $X$ and $Y$ as we build the %monochromatic 
cliques $A$ and $B$. 

\begin{figure}[t]
\centering
\includegraphics[width=0.48\textwidth]{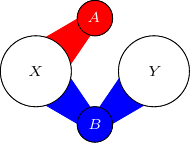}
\caption{The setting of the proof of Theorem~\ref{thm:GNNW}.}
\label{fig:book:alg}
\end{figure}

The first step is to choose an initial pair of sets $X$ and $Y$ that have many blue edges between them. If the density of blue edges is high enough, then we simply do so by choosing a random bipartition of the vertices; if not, then we choose a vertex $v$ of maximum red degree, and work instead inside $N_R(v)$, with $\ell$ replaced by $\ell - 1$. 

To be slightly more precise, Gupta, Ndiaye, Norin and Wei perform this step using the following induction hypothesis:
\begin{equation}\label{eq:GNNW:induction}
R(\ell,k) \le 4\gigl( k + \ell \gigr) \bigg( \frac{k + 2\ell}{k} \bigg)^{k/2} p^{-\ell}
%R(\ell,k) \le 4\gigl( k + \ell \gigr) \bigg( 1 - \bigg( \frac{\sqrt{5} + 1}{2} \bigg) p \bigg)^{-k/2} p^{-\ell}
\end{equation}
%$$R(\ell,k) \le 4(k+\ell) \bigg( \frac{(\sqrt{5} + 1)(k + 2\ell)}{4\ell} \bigg)^\ell \bigg( \frac{k + 2\ell}{k} \bigg)^{k/2}$$
for every $k,\ell \in \N$ with $\ell \le k$, where
$$p = \frac{4}{\sqrt{5} + 1} \bigg( \frac{\ell}{k + 2\ell} \bigg).$$
By the induction hypothesis, we may assume that every vertex has red degree at most $pn$, and hence there exists a partition $V(K_n) = X \cup Y$ such that the density of blue edges between $X$ and $Y$ is at least $1 - p$. 

The idea is now to show that we can find either a blue copy of $K_k$ in $X \cup Y$, or a red copy of $K_\ell$ inside either $X$ or $Y$. As in the proof of~\eqref{eq:ESz:weaker}, we do so by choosing one vertex $x \in X$ in each step, moving it to either $A$ or $B$, and shrinking the sets $X$ and $Y$. However, perhaps surprisingly, in either case we replace $Y$ by $N_B(x) \cap Y$, the \emph{blue} neighbourhood of $x$. To be more precise, in each step of the algorithm we make one of the following moves:\vskip0.1cm 
\begin{itemize}
\item[$(a)$] add $x$ to $A$ and update $X \to N_R(x) \cap X$ and $Y \to N_B(x) \cap Y$, or\vskip0.1cm
\item[$(b)$] add $x$ to $B$ and update $X \to N_B(x) \cap X$ and $Y \to N_B(x) \cap Y$.\vskip0.1cm
\end{itemize}
The motivation behind this is that we are happy in case $(b)$ unless the density of blue edges between $N_B(x) \cap X$ and $N_B(x) \cap Y$ is significantly lower than between $X$ and $Y$, and if that happens then the density of blue edges between $N_R(x) \cap X$ and $N_B(x) \cap Y$ must be significantly higher, which `pays' for the loss in the size of $Y$. 

The beautiful innovation of Gupta, Ndiaye, Norin and Wei is that when running this algorithm, it is sufficient to track only the `excess' number of blue edges between $X$ and $Y$ above some fixed density $q$. That is, they show that if 
$$f_q(X,Y) = e_B(X,Y) - q|X||Y|$$
is at least a certain quantity (depending on $q$), then we can find one of the monochromatic cliques that we are looking for. In fact, for the induction hypothesis we need a slightly more general statement, since in the middle of the algorithm we are looking for a blue clique of size $k - |B|$ in $X \cup Y$, a red clique of size $\ell - |A|$ in $X$, or a red clique of size $\ell$ in $Y$. 

\begin{lemma}\label{lem:GNNW}
Let $X$ and $Y$ be disjoint sets, let $0 < \gamma < q < 1$, and let $k,\ell,m \in \N$. If
$$f_q(X,Y) \ge \gigl( k + m \pigr) \gamma^{-k} \gigl( 1 - \gamma \gigr)^{-\ell} \gigl( q - \gamma \gigr)^{-m},$$
then there exists either a red $K_m$ in $X$, a red $K_\ell$ in $Y$, or a blue $K_k$ in $X \cup Y$. 
\end{lemma}

The proof of this lemma is now straightforward. We first choose $x \in X$ so that  
$$f_q\big( X, N_B(x) \cap Y \big) \ge q \cdot f_q(X,Y),$$
which is possible by a simple convexity argument. Now, if 
$$f_q\big( N_B(x) \cap X, N_B(x) \cap Y \big) \ge \bigg( \frac{k+m-1}{k+m} \bigg) \cdot \gamma \cdot f_q(X,Y),$$
then we make move $(b)$, and apply the induction hypothesis. Similarly, if 
$$f_q\big( N_R(x) \cap X, N_B(x) \cap Y \big) \ge \bigg( \frac{k+m-1}{k+m} \bigg) \pigl( q - \gamma \pigr) \cdot f_q(X,Y),$$
then we make move $(a)$, and apply the induction hypothesis to complete the proof. The only remaining possibility is that $x$ has at least $\frac{1}{k+m} \cdot f_q(X,Y)$ neighbours in $Y$. But by our bound on $f_q(X,Y)$ this implies that $|Y| \ge R(\ell,k)$, and hence we can find either a red copy of $K_\ell$ or a blue copy of $K_k$ in $Y$, as required. Applying Lemma~\ref{lem:GNNW} with $q = 1 - p$ and $\gamma = 1 - \big( \frac{\sqrt{5}+1}{2} \big)$, so $p^2 = (1-\gamma)(q - \gamma)$, gives~\eqref{eq:GNNW:induction}, which then implies~\eqref{eq:GNNW:thm} for $\ell \ll k$.

To finish this section, let us state the following conjecture, which says that the bound given by Theorem~\ref{thm:GNNW} is still super-exponentially far from the truth. 

\begin{conjecture}\label{conj:off:diagonal}
For every fixed $C > 0$, we have
%\begin{equation}\label{eq:off:diagonal}
$$R(\ell,k) \le e^{-C \ell} {k + \ell \choose \ell}$$
%\end{equation}
for all sufficiently large $k,\ell \in \N$ with $\log k \ll \ell \ll k$. 
\end{conjecture}

It seems that a proof of Conjecture~\ref{conj:off:diagonal} would require a significant new idea. 

\section{An improved lower bound near to the diagonal}\label{MSX:sec}

How far is Theorem~\ref{thm:GNNW} from the lower bound? If we take the red edges to be a copy of $G(n,p)$, then a standard application of the Lovász Local Lemma (as in~\cite{S77}) implies that
\begin{equation}\label{eq:lower:LLL}
R(\ell,k) \ge \bigg( \frac{k}{\ell \cdot \log(k/\ell)} \bigg)^{(\ell+1)/2},
\end{equation}
for all $1 \ll \ell \ll k$. When $\ell = \Theta(k)$, however, the bound given by the local lemma is only a constant factor stronger than that given by a simple 1st moment argument:
\begin{equation}\label{eq:lower:random}
R(\ell,k) \ge p^{-\ell/2} \qquad \text{where} \qquad \frac{k}{\ell} = \frac{\log p}{\log(1-p)}.
\end{equation}
In particular, note that if $\ell = k$ then this reduces to Erd\H{o}s' bound $R(k) \ge 2^{-k/2}$.  

In a significant breakthrough, the bound given by $G(n,p)$ was finally improved earlier this year by Ma, Shen and Xie~\cite{MSX}. More precisely, for all pairs $(\ell,k)$ with $k/\ell$ equal to a constant greater than $1$, they improved the bound~\eqref{eq:lower:random} by an exponential factor. 

\begin{theorem}[Ma, Shen and Xie, 2025+]\label{thm:MSX}
For each $\lambda > 1$, there exists $\eps = \eps(\lambda) > 0$ such that the following holds. If $\ell,k \in \N$ are sufficiently large and $k = \lambda \ell$, then 
\begin{equation}\label{eq:MSX}
R(\ell,k) \ge \pigl( p + \eps \pigr)^{-\ell/2} \qquad \text{where} \qquad \frac{k}{\ell} = \frac{\log p}{\log(1-p)}.
\end{equation}
\end{theorem}

Like many of the constructions that we have seen in this survey, the colouring that Ma, Shen and Xie used to prove Theorem~\ref{thm:MSX} is surprisingly simple to define -- in fact, it is quite similar to Erd\H{o}s' first lower bound on $R(3,k)$ (see Section~\ref{sec:Erdos:geometry}). The difficult part is to show (or even to guess) that it works!

To define the colouring, fix $d \in \N$, and for each set $A \subset \{-1,1\}^d$ and $\alpha \in [-d,d]$, consider a red-blue colouring of the complete graph with vertex set $A$, in which the edges 
$$\big\{ uv : \< u,v \> < \alpha \big\}$$
are coloured red, and the remaining edges are coloured blue. We will consider this colouring with $d = Ck^2$ for some large constant $C > 0$, with $\alpha = -c\sqrt{d}$ for some constant $c > 0$, and with the set $A$ chosen uniformly at random from the subsets of $\{-1,1\}^d$ of size $n$. 

Let $p$ be the probability that a given edge is red, and note that $p < 1/2$ is a constant depending on $c$. What is the expected number of monochromatic cliques in this colouring? It is not difficult to see (or at least to guess) that the events $\{ uv \text{ is red} \}$ are negatively correlated, and that therefore the probability that a fixed set of $\ell$ vertices forms a red clique should be less than $p^{\ell \choose 2}$. On the other hand, the events $\{ uv \text{ is blue} \}$ are positively correlated, and hence each set of $k$ vertices forms a blue clique with probability greater than $\pigl( 1 - p \pigr)^{k \choose 2}$. The optimal value of $p$ will therefore be slightly larger than the one used to prove~\eqref{eq:lower:random}. 

How do the sizes of these two effects compare? This is a much trickier question, but perhaps we can get some intuition by thinking about the case in which $p$ is small (so $c$ is large). The force of the negative correlation is then large, since every pair must have an unusually large negative inner product. On the other hand, the positive correlation will be relatively small, since the average inner product of a pair is only a little larger than zero. We might therefore hope that the decrease in the size of the largest red clique `outweighs' the increase in the size of the largest blue clique, compared with the random graph $G(n,p)$. This is exactly what Ma, Shen and Xie show, not only when $c$ is large, but for every $c > 0$.  

In order to perform the intricate calculations in the proof, Ma, Shen and Xie found it more convenient to consider a continuous version of the construction described above, in which the elements of the set $A$ are chosen uniformly and independently at random from the unit sphere\footnote{The proof of Theorem~\ref{thm:MSX} has recently been simplified by Hunter, Milojević and Sudakov~\cite{HMS} and Sahasrabudhe~\cite{S25} by instead choosing points in $\R^d$ according to a Gaussian distribution.} in~$\R^d$. %, and the threshold $\alpha$ was chosen so that each edge is red with probability $q$. 
%let $G_{n,q}(\R^d)$ denote the random colouring obtained by choosing $n$ points uniformly and independently at random from the unit sphere in $\R^d$, and choosing the threshold $\alpha$ so that each edge is red with probability $q$. 
Such \emph{random geometric graphs} have a long history in extremal and probabilistic combinatorics, beginning %in the 1970s 
with the famous construction of Bollobás and Erd\H{o}s~\cite{BE76} that gives a sharp lower bound on the Ramsey--Turán number of $K_4$, and they are also important objects in probability theory, see for example~\cite{BDER,DGLU,Pen}. However, before the proof of Theorem~\ref{thm:MSX} their potential for proving lower bounds on Ramsey numbers had not been appreciated, and it does not seem unreasonable to hope that they may have many further applications in Ramsey theory.

% Ma, Shen and Xie show that there exists a value of $q$ (slightly larger than the density $p$ used in~\eqref{eq:lower:random}) such that the expected number of both red copies of $K_\ell$ and blue copies of $K_k$ in $G_{n,q}(\R^d)$ is smaller (by an exponential factor) than in the colouring given by $G(n,p)$. 

%simpler proof by Hunter, Milojevic and Sudakov, and also Julian

\section{Diagonal Ramsey numbers}\label{diag:sec}

The method outlined in Section~\ref{closer:sec} can be extended, with a number of additional ideas, to prove Theorem~\ref{thm:diagonal}, which gives an exponential improvement for the diagonal Ramsey numbers $R(k)$. However, the proof given by this approach is for several reasons rather unsatisfying: it requires a long and complicated calculation to check that it really improves the Erd\H{o}s--Szekeres bound, and doesn't provide a nice, simple story for why it does better. The approach moreover gives a worse bound than the Erd\H{o}s--Szekeres algorithm for the multicolour diagonal Ramsey numbers $R_r(k)$, the smallest $n \in \N$ such that every $r$-colouring of the edges of $K_n$ contains a monochromatic copy of $K_k$.

In this section we will outline a second proof of Theorem~\ref{thm:diagonal}, which was found by Campos, Griffiths, Morris and Sahasrabudhe (the authors of the original proof~\cite{CGMS}) together with Balister, Bollobás, Hurley and Tiba~\cite{BBCGHMST}, that \emph{does} extend to the multicolour setting. This second proof moreover has various other advantages over the original: it is much shorter, it provides a clear story for why it improves the Erd\H{o}s--Szekeres bound, and it is based on a natural geometric lemma that has a surprisingly simple and elegant proof.

\begin{theorem}[Balister, Bollobás, Campos, Griffiths, Hurley, Morris, Sahasrabudhe and Tiba, 2024+]
\label{thm:multicolour}
For each $r \ge 2$, there exists $\delta = \delta(r) > 0$ such that 
$$R_r(k) \le e^{-\delta k} r^{rk}$$ 
for all sufficiently large $k \in \N$. 
\end{theorem}

The setting of the proof of Theorem~\ref{thm:multicolour} is illustrated in Figure~\ref{fig:multicolour:alg}; as before, $X$ is our `reservoir' set, and for each $i \in [r]$ we build a clique $A_i$ in colour $i$. However, we now also build a `book' $(A_i,Y_i)$ in each colour. Here we say that $(A,Y)$ is a red \emph{$(t,m)$-book} if $|A| = t$ and $|Y|= m$, and every edge with one endpoint in $A$ and the other in $A \cup Y$ is red. 

\begin{figure}[t]
\centering
\includegraphics[width=0.45\textwidth]{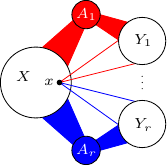}
\caption{The setting of the Multicolour Book Algorithm.}
\label{fig:multicolour:alg}
\end{figure}

%\smallskip
%\pagebreak

For simplicity, let us focus for a moment on the case $r = 2$; the approach in the general case is essentially the same. Our plan is to find a monochromatic copy of $K_k$ (in an arbitrary red-blue colouring of $E(K_n)$) by first finding a monochromatic $(t,m)$-book, where
$$t \ge \delta^4 k \qquad \text{and} \qquad m \ge e^{-\delta t^2 / k} \, 2^{-t} n \ge R(k-t,k)$$
for some (small) constant $\delta > 0$. Note that in the set of size $m$ we must have either a copy of $K_{k-t}$ in the same colour as the book, or a copy of $K_k$ in the other colour, and in either case we obtain a monochromatic copy of $K_k$, as required. Moreover, since we have
$$R(k-t,k) \le {2k - t \choose k - t} \le e^{-t^2 / 6k} \, 2^{2k-t}$$
by the Erd\H{o}s--Szekeres bound~\eqref{eq:ESz:bound}, this will suffice to prove Theorem~\ref{thm:multicolour} when $r = 2$. The following lemma provides us with such a book. 

\begin{lemma}\label{lem:book}
Let\/ $c$ be an\/ $r$-colouring of\/ $E(K_n)$, and let\/ $X,Y_1,\ldots,Y_r \subset V(K_n)$. For every $p > 0$ and $k,m \in \N$, the following holds for some $t \ge \delta^4 k$. If 
%\begin{equation}\label{eq:book:nbhds}
$$|N_i(x) \cap Y_i| \ge p|Y_i|$$ 
%\end{equation}
for every $x \in X$ and every colour $i \in [r]$, and moreover
$$|X| \ge \bigg( \frac{2}{p} \bigg)^{\delta k} \qquad \text{and} \qquad \min\big\{ |Y_1|, \ldots, |Y_r| \big\} \ge 2^{\delta t^2/k} p^{-t} m,$$ 
then $c$ contains a monochromatic $(t,m)$-book.
\end{lemma}

%\begin{lemma}\label{lem:book}
%Let\/ $\chi$ be an\/ red-blue colouring of\/ $E(K_n)$, and let\/ $X,Y,Z \subset V(K_n)$. For every $p > 0$ and every $t,m \in \N$ with $t = \delta^4 k$, the following holds. If 
%$$|N_R(x) \cap Y| \ge p|Y| \qquad \text{and} \qquad |N_B(x) \cap Z| \ge p|Z|$$
%for every $x \in X$, and moreover
%$$|X| \ge \bigg( \frac{2}{p} \bigg)^{\delta k} \qquad \text{and} \qquad \min\big\{ |Y|, |Z| \big\} \ge 2^{\delta t^2/k} p^{-t} m,$$ 
%then $\chi$ contains a monochromatic $(t,m)$-book.
%\end{lemma}

To deduce Theorem~\ref{thm:multicolour} from Lemma~\ref{lem:book}, we need to find sets $X$, and $Y_1,\ldots,Y_r$ satisfying the conditions of the lemma with $p \approx 1/r$. To do so, we simply run Erd\H{o}s--Szekeres steps (always choosing the vertex of maximum degree in one colour) until every vertex has roughly an equal number of neighbours of each colour, and then partition randomly. % (cf.~Section~\ref{closer:sec}). 

To prove Lemma~\ref{lem:book}, in each step we either find a vertex of $X$ that can be added to one of the sets $A_i$ without significantly decreasing the density of colour $i$ edges between $X$ and $Y_i$, or we find a `density boost': large subsets $X' \subset X$ and $Y' \subset Y_i$ such that the density of colour $i$ edges between $X'$ and $Y'$ is significantly higher than that between $X$ and $Y_i$. That we can do so is a consequence of the following key geometric lemma.

\begin{lemma}\label{lem:lambda}
Let\/ $U$ and\/ $U'$ be i.i.d.~random variables taking values in a finite set~$X$, and let $f_1,\ldots,f_r \colon X \to \R^n$ be arbitrary functions. Either 
\begin{equation}\label{eq:colour:step}
\Pr\digl( \pigl\langle f_i(U), f_i(U') \pigr\rangle \ge -1 \, \text{ for all } \, i \in [r] \digr) \hspace{-0.01cm} \ge \hspace{0.03cm} \delta
\end{equation}
or there exist a colour $i \in [r]$ and a sufficiently large $\lambda > 0$ such that
\begin{equation}\label{eq:boost:step}
\Pr\digl( \pigl\langle f_i(U),f_i(U') \pigr\rangle \ge \lambda \digr) \ge e^{-O(\sqrt{\lambda})}.
\end{equation}
\end{lemma}

Roughly speaking, this lemma says that if the $r$ functions exhibit a large amount of `negative correlation', then one of them must exhibit a significant amount of `clustering'. In our application, the function $f_i$ encodes the colour $i$ neighbourhoods in the set $Y_i$ of the vertices of $X$, and $U$ and $U'$ are uniformly-chosen elements of $X$. If~\eqref{eq:colour:step} holds, then we choose a vertex $x \in X$ and a colour $i \in [r]$ such that the set %there are at least $\delta |X| / r$ vertices $y \in X$ for which
$$X' = \big\{ \hspace{0.02cm} y \in X : \pigl\langle f_i(x), f_i(y) \pigr\rangle \ge -1 \, \text{ and } \, c(xy) = i \hspace{0.02cm}\big\},$$
has size at least $\delta |X| / r$, and update the sets as follows:
$$X \to X', \qquad Y_i \to N_i(x) \cap Y_i \qquad \text{and} \qquad A_i \to A_i \cup \{x\}.$$
On the other hand, if~\eqref{eq:boost:step} holds, then we instead choose a vertex $x \in X$ such that the set 
$$X' = \big\{ \hspace{0.02cm} y \in X : \pigl\langle f_i(x), f_i(y) \pigr\rangle \ge \lambda \hspace{0.02cm}\big\},$$
has size at least $e^{-O(\sqrt{\lambda})} |X|$, and update the sets as follows:
$$X \to X' \qquad \text{and} \qquad Y_i \to N_i(x) \cap Y_i.$$
The bounds on the inner product guarantee that in the first case the density of colour $i$ edges between $X$ and $Y_i$ does not decrease too much, and in the second case that it increases substantially. Note that in the second case the set $X$ may shrink by a large factor, but since the factor $e^{-O(\sqrt{\lambda})}$ is a sub-exponential function of $\lambda$, this does not cost us too much.  

Finally, let us briefly discuss the (surprisingly simple) proof of Lemma~\ref{lem:lambda}. The key idea is to define the following function: 
\begin{equation}\label{def:f}
g(x_1,\dots,x_r) = \sum_{j = 1}^r x_j \prod_{i \ne j} \big( 2 + \cosh\sqrt{x_i} \big),
\end{equation}
where we define $\cosh \sqrt{x}$ via its Taylor expansion
$$\cosh\sqrt{x} = \sum_{n = 0}^\infty \frac{x^n}{(2n)!}.$$ 
In particular, all of the coefficients of the Taylor expansion of $g$ are non-negative, which implies that 
$$\Ex\rigl[ g\sigl( \pigl\< f_1(U,U') \pigr\>, \ldots, \pigl\< f_r(U,U') \pigr\> \sigr) \rigr] \ge 0,$$
since the moments of the inner products $\pigl\langle f_i(U),f_i(U') \pigr\rangle$ are all non-negative. The lemma now follows from a straightforward calculation, using the following inequalities: 
%function $g$ also satisfies the following inequalities. 
$$
g(x_1,\dots,x_r) \le \left\{\begin{array}{cl}
3^r r \exp\bigg( \displaystyle\sum_{i = 1}^r \sqrt{ x_i + 3r } \bigg) \quad & \text{if } \,\, x_i \ge - 3r \,\text{ for all }\, i \in [r];\\[+3ex]
-1 & \text{otherwise.} 
\end{array} \right.
$$

The proof in~\cite{BBCGHMST} implies that Theorem~\ref{thm:multicolour} holds with $\delta$ a polynomial function of $r$. A natural next aim would be to prove it for an absolute constant $\delta$. 

\begin{conjecture}\label{conj:diagonal}
There exists a constant $\delta > 0$ such that 
$$R_r(k) \le e^{-\delta k} r^{rk}$$ 
for all $r \ge 2$ and all sufficiently large $k \in \N$. 
\end{conjecture}

The best-known lower bounds on $R_r(k)$ are of the form $c^{rk}$ for some constant $c > 1$. The first such bound was proved by Abbott~\cite{A72} in 1972, and the value of $c$ was improved recently, first by Conlon and Ferber~\cite{CF}, and subsequently by Wigderson~\cite{Wig} and Sawin~\cite{Saw}. 

At the opposite end of the spectrum, the problem is also wide open in the case $k = 3$. The best known upper bound is of the form $R_r(3) = O(r!)$, which follows from the Erd\H{o}s--Szekeres algorithm, and was originally proved by Schur~\cite{Schur} in 1911. Any improvement of this bound would be extremely welcome. 

\begin{problem}\label{prob:upper:Rr3}
Show that
$$R_r(3) = o(r!)$$ 
as $r \to \infty$.
\end{problem}

A much more daunting task would be to solve the following famous problem of Erd\H{o}s~\cite{E81}. % (see, e.g.,~\cite{NR}).

\begin{problem}[Erd\H{o}s, 1970s\footnote{Ne\v{s}et\v{r}il and Rosenfeld~\cite{NR} mention that in 1974 this was already ``one of the `prized' Erd\H{o}s problems". However, the earliest paper that we were able to find in which the problem is stated in this form is~\cite{E81}.}]\label{prob:Erdos:Rr3}
Does there exists a constant $C > 0$ such that 
$$R_r(3) \le 2^{Cr}$$ 
for all $r \in \N$?
\end{problem}
 
 %\footnote{This problem is described by Chung and Graham~\cite{CG} as ``a very old problem of Erd\H{o}s". However, the earliest paper that we were able to find in which the problem is stated in this form is~\cite{E81}.}
 
The best known lower bounds on $R_r(3)$ are obtained via the inequality $R_r(3) \ge S(r)$, where $S(r)$ denotes the $r$th Schur number: the smallest $n \in \N$ such every $r$-colouring of the set $[n]$ contains a monochromatic solution of the equation $x + y = z$. We refer the reader to~\cite{NR} for a well-written and entertaining history of bounds on $S(r)$ and $R_r(3)$. 
 
%We will in fact prove the following quantitative version of the theorem.

%\begin{theorem}\label{thm:Ramsey:multicolour:quant}
%Let $r \ge 2$, and set $\delta = 2^{-160} r^{-12}$. Then
%$$R_r(k) \le e^{-\delta k} r^{rk}$$ 
%for every $k \in \N$ with $k \ge 2^{160} r^{16}$. 
%\end{theorem}

\section{Induced Ramsey numbers}\label{induced:sec}

In this final section we will provide a rough sketch of the amazing recent breakthrough of Aragão, Campos, Dahia, Filipe and Marciano~\cite{ACDFM} on induced Ramsey numbers. Here (like in Section~\ref{diag:sec}) we will work in the more general setting of $r$-colourings, so let us write  
$$G \xxrightarrow[r]{\mathrm{ind}} H$$
if every $r$-colouring of $E(G)$ contains a monochromatic induced copy of~$H$, and define 
$$R_{\hspace{0.02cm} r}^{\mathrm{ind}}(H) = \min\big\{ v(G) : G \xxrightarrow[r]{\mathrm{ind}} H \big\}.$$
These numbers were shown to be finite for every $r$ and every graph $H$ in~\cite{D75,EHP,R73}. %by Deuber~\cite{D75}, Erd\H{o}s, Hajnal and Posa~\cite{EHP} and Rödl~\cite{R73}, %but the best bound given by these proofs was double-exponential in $k$. Despite this, 
Not long afterwards, Erd\H{o}s~\cite{E75,E84} made the following influential conjecture.

\begin{conjecture}[Erd\H{o}s, 1975]\label{conj:induced}
There exists a constant $C > 0$ such that
%\begin{equation}\label{eq:conj:induced}
$$R_{\hspace{0.02cm} 2}^{\mathrm{ind}}(H) \le 2^{Ck}$$
%\end{equation}
 for every graph $H$ with $k$ vertices.
 \end{conjecture}
 
 %Over the subsequent decades this has become one of the most notorious conjectures in Ramsey theory, with numerous significant advances by some of the leading researchers in the field. 
 
The first single-exponential bound on $R_{\hspace{0.02cm} 2}^{\mathrm{ind}}(H)$ was obtained by Kohayakawa, Pr\"omel and R\"odl~\cite{KPR}, who used a random graph built using projective planes to show that
\begin{equation}\label{eq:KPR}
R_{\hspace{0.02cm} 2}^{\mathrm{ind}}(H) \le k^{O(k \log k)}
\end{equation}
for every graph $H$ with $k$ vertices. An alternative approach for arbitrary pseudorandom graphs was introduced by Fox and Sudakov~\cite{FS08,FS09}, who gave a second proof of~\eqref{eq:KPR}, and also obtained the first reasonable bound in the case $r > 2$, showing that
\begin{equation}\label{eq:FS:BS}
R_{\hspace{0.02cm} r}^{\mathrm{ind}}(H) \le r^{O(r k^2)}
\end{equation}
for every graph $H$ with $k$ vertices and every $r \in \N$. This method was then developed further by Conlon, Fox and Sudakov~\cite{CFS12}, who improved the bound~\eqref{eq:KPR} to 
$$R_{\hspace{0.02cm} 2}^{\mathrm{ind}}(H) \le k^{O(k)}.$$
%Their approach moreover has many further applications in graph Ramsey theory (see~\cite{CFS}). 
More recently, another proof of~\eqref{eq:FS:BS} was found by Balogh and Samotij~\cite{BS}, who used their `efficient' container lemma to show that $G(n,1/2) \xxrightarrow[r]{\mathrm{ind}} H$ with high probability. 
  
Conjecture~\ref{conj:induced} was finally proved by Aragão, Campos, Dahia, Filipe and Marciano~\cite{ACDFM}, who moreover resolved the problem for all $r \ge 2$.
  
\begin{theorem}[Aragão, Campos, Dahia, Filipe and Marciano, 2025+]\label{thm:induced:multi}
There exists an absolute constant $C > 0$ such that\vskip-0.38cm
\begin{equation}\label{eq:multicolour:induced}
R_{\hspace{0.02cm} r}^{\mathrm{ind}}(H) \le r^{C r k}
\end{equation}
for every $r \ge 2$ and every graph $H$ with $k$ vertices. 
\end{theorem}

This bound is close to best possible, since $R_{\hspace{0.02cm} r}^{\mathrm{ind}}(K_k) = R_r(k)$, and the bound~\eqref{eq:multicolour:induced} matches the best-known upper bound on $R_r(k)$ up to the value of the constant $C$ (cf.~Section~\ref{diag:sec}).

Aragão, Campos, Dahia, Filipe and Marciano actually proved the following stronger theorem, which moreover implies that for almost all graphs $G$ with $n \ge r^{Crk}$ vertices, every $r$-colouring of $E(G)$ contains an induced monochromatic copy of \emph{every} graph $H$ on $k$ vertices.

\begin{theorem}\label{thm:induced:multi:Gnp}
Let $H$ be a graph with $k$ vertices, let $r \ge 2$, and let $n \ge r^{Crk}$. Then
$$G(n,1/2) \xxrightarrow[r]{\mathrm{ind}} H$$
with probability at least $1 - \exp\pigl( - \delta n^2 \pigr)$, where $\delta = r^{-Crk}$.
\end{theorem}

The bound on the probability in Theorem~\ref{thm:induced:multi:Gnp} is also close to best possible, since if $G(n,1/2)$ has chromatic number less than $R_r(k)$ then its edges can be $r$-coloured without creating a monochromatic copy of $K_k$, and this occurs with probability at least $2^{-n^2/R_r(k)}$. 

We will next attempt to give a high-level overview of the (extremely complicated) proof of Theorem~\ref{thm:induced:multi:Gnp}. To set the scene, consider the following naive attempt to find a copy\footnote{To avoid repetition, we will write ``copy of $H$" to mean ``induced monochromatic copy of $H$".} of $H$ using an Erd\H{o}s--Szekeres-type algorithm: % (cf.~the proof of Theorem~\ref{thm:ESz:bound}). 
apply the induction hypothesis inside a set $U \subset V(G)$ to find a copy of $H - v$ (the graph obtained from $H$ by removing a vertex $v$), and then attempt to use the edges between $U$ and $V(G) \setminus U$ to extend it to a copy of $H$. 

The reader will perhaps already have noticed a number of potential problems with this approach. Most obviously, if we only find one copy of $H - v$ (in red, say) then we can easily avoid extending it to a red copy of $H$, simply by not using the colour red for any of the edges between $U$ and $V(G) \setminus U$. Dealing with this problem is easy, however: if we generalise to the off-diagonal setting (in which we aim to find a copy of $H_i$ in colour $i$), then we can use the induction hypothesis to find a colour $i$ copy of $H_i - v$ in $U$ for each $i \in [r]$.

A seemingly more catastrophic problem is that the enemy is allowed to colour the edges inside $U$ after seeing \emph{all} of the edges of $G \sim G(n,1/2)$, including those outside $U$. In particular, this means that the colouring of the edges inside $U$ will affect (perhaps significantly) the distribution of the remaining edges. In order to deal with this problem, we are forced take a union bound over the roughly $r^{|U|^2}$ choices of the colouring inside $U$. To reduce the pain of this union bound, we would like to take $U$ as small as possible; for the induction hypothesis to apply, however, we cannot take it to be smaller than $r^{-Cr} n$. 

We are now left with the task of showing that for each choice of the colouring inside $U$, the probability that we fail to extend to a copy of $H$ is smaller than $r^{-|U|^2}$. But this seems hopeless: the probability that there are \emph{zero} edges between a copy of $H - v$ and $V(G) \setminus U$ is at least $2^{-kn}$, which is already much too large, and the probability that it fails to extend to a (not necessarily monochromatic) copy of $H$ is even larger: roughly $(1 - 2^{-k})^{n}$. 

This suggests that we need to strengthen the induction hypothesis so that, instead of a single copy, we find \emph{many} copies of $H_i - v$ in $U$ for each colour $i$. In fact, even this turns out not to be enough: these copies must also be sufficiently `well-distributed' (for example, the copies should not all intersect a subset of $U$ of size $o(n)$, since a set of this size has no neighbours outside $U$ with probability $2^{-o(n^2)}$). To make this precise, Aragão, Campos, Dahia, Filipe and Marciano introduced the following key definition.

\begin{defn}[$(p,R)$-Janson hypergraphs]
We say that a hypergraph $\cH$ is $(p,R)$-Janson if there exists a probability measure $\mu$ supported on the edges of $\cH$ such that
$$\sum_{\substack{L \hspace{0.01cm} \subset \hspace{0.01cm} V(\cH) \\ |L| \hspace{0.01cm} \ge \hspace{0.01cm} 2}} p^{-|L|} \bigg( \sum_{L \hspace{0.01cm}\subset \hspace{0.01cm} E \hspace{0.01cm} \in \hspace{0.01cm} \cH} \mu(E) \bigg)^2 < \frac{1}{R}.$$
\end{defn}

They apply this definition to the hypergraph $\cH$ with vertex set $U$ and edge set 
$$\big\{ S \subset U : G[S] \text{ is a copy of $H_i - v$ in colour $i$} \big\},$$
and the induction hypothesis tells us that this hypergraph is $(p,p|U|)$-Janson for some $p$ (a polynomial function of $k$ and $r$). We now want to prove the following lemma, which is a simplified (and slightly imprecise) version of~\cite[Lemma~3.1]{ACDFM}.

\begin{lemma}\label{lem:simple:induced:key}
If $\cH$ is $(p,p|U|)$-Janson, then the probability that there exists a set of $|U|/4r$ edges between $u$ and $U$ that extend no edge of $\cH$ to a copy of $H_i$ is at most $2^{-\Omega(|U|)}$. 
\end{lemma}

Here we think of the $|U|/4r$ edges as being colour $i$, and we are trying to extend to an induced copy of $H_i$ in colour $i$, so the neighbourhood of $u$ in an edge of $\cH$ must exactly match that of the vertex $v$ in $H_i$, and all of the edges must have colour $i$.  

Aragão, Campos, Dahia, Filipe and Marciano proved Lemma~\ref{lem:simple:induced:key} using the method of hypergraph containers, which is a generalisation of the method of graph containers (see Section~\ref{sec:R4k}, where we used graph containers to prove a lower bound on $R(4,k)$). We refer the reader to the survey~\cite{ICM} for background on hypergraph containers. More precisely, they used an `efficient' container lemma of Campos and Samotij~\cite{CS}, which gives much better dependence on the uniformity of the hypergraph than the original container lemmas from~\cite{BMS,ST}. The first efficient container lemma was developed by Balogh and Samotij~\cite{BS}, who used it (in a much simpler way) to give a new proof of the bound~\eqref{eq:FS:BS}.  

Unfortunately, however, Lemma~\ref{lem:simple:induced:key} is not strong enough for our purposes, since we now need not only one copy of $H_i$, but a $(p,pn)$-Janson collection of copies! The actual lemma we need (see~\cite[Lemma~5.1]{ACDFM}) is roughly as follows. Suppose that $\cH$ is $(p,p|U|)$-Janson, and that we have already constructed a $(p,R)$-Janson family of copies of $H_i$ in colour $i$. Then the probability that there is a set of $|U|/4r$ edges between $u$ and $U$ that does not extend this collection to a $(p,R+1)$-Janson family of copies of $H_i$ is at most $2^{-\Omega(|U|)}$.

The proof of this lemma is the most difficult and novel part of the proof of Theorem~\ref{thm:induced:multi:Gnp}, and involves an exciting new generalisation of the method of hypergraph containers. In order to motivate this approach, let us briefly recall the classical hypergraph container method, as introduced in~\cite{BMS,ST} and then strengthened in~\cite{BS}. Roughly speaking, given a $k$-uniform hypergraph $\cH$ whose edges are reasonably `uniformly' distributed, the container method provides a (not too large) family $\cC$ of `almost independent' sets (meaning that they contain at most $\eps \cdot e(\cH)$ edges of $\cH$) that cover the independent sets of $\cH$. The size of the family $\cC$ depends on how uniformly the edges are distributed, and also on $\eps$, and on the uniformity $k$. The power of this lemma comes from the fact that we can now take a union bound over the `containers' $C \in \cC$, and deal with each container using a suitable supersaturation theorem.  

To be more precise, let $\cH$ be a $k$-uniform hypergraph with $n$ vertices, and let $\Delta_\ell(\cH)$ denote the maximum over $\ell$-sets $L$ of the number of edges of $\cH$ that contain $L$. If 
$$\Delta_\ell(\cH) = O\bigg( \tau^{\ell-1} \cdot \frac{e(\cH)}{n} \bigg)$$
for every $1 \le \ell \le k$, then there exists a family of `containers' $\cC$, with
$$|\cC| \le \exp\sigl( \hspace{0.02cm} K(k,\eps) \cdot \tau n \log n \sigr),$$
such that every independent set $I \in \cI(\cH)$ is a subset of some container $C \in \cC$, and each $C \in \cC$ contains at most $\eps \cdot e(\cH)$ edges of $\cH$. To prove this statement, we use a deterministic algorithm to find, inside each independent set $I \in \cI(\cH)$, a small `fingerprint' $f(I)$ with the property that the container of $I$ is determined by $f(I)$. 

The original container theorem~\cite{BMS,ST} gave a function $K$ with an optimal dependence on $\eps$, but a fairly poor (super-exponential) dependence on $k$. The efficient container lemma of Balogh and Samotij~\cite{BS} reduced this to a polynomial dependence, and Campos and Samotij~\cite{CS} gave two simple and elegant proofs of this statement, together with several generalisations. In particular, they proved the following container lemma, which plays a crucial role in the proof of Theorem~\ref{thm:induced:multi:Gnp}. 

\begin{lemma}[Campos and Samotij, 2024+]\label{lem:CS}
Let $\cH$ be a hypergraph with $n$ vertices, and let $0 < p \le \delta < 1$. There exists a family $\cT$ of subsets of $V(\cH)$, and a function $f \colon \cI(\cH) \to \cT$, %where $\cI(\cH)$ denotes the family of independent sets of $\cH$, 
such that the following hold:
\begin{itemize}
\item[$(a)$] $f(I) \subset I$ for every $I \in \cI(\cH)$.\smallskip
\item[$(b)$] $|T| \le pn/\delta$ for every $T \in \cT$.\smallskip
\item[$(c)$] For each $T \in \cT$, there is a hypergraph $\cG_T$ with vertex set $V(\cH) \setminus T$ that covers $\cH$, and satisfies
\begin{equation}\label{eq:CS:prop}
\Pr\big( S \subset V_q \;\pig|\; V_q \in \cI\pigl( \cG_T \pigr) \big) > \pigl( 1 - \delta \pigr)^{|S|} q^{|S|}
\end{equation}
for all $S \not\in \cG_T$. Moreover, $I \in \cI(\cG_T)$ for every $I \in \cI(\cH)$ such that $f(I) = T$. 
\end{itemize}
\end{lemma}

Note in particular that in Lemma~\ref{lem:CS} we do not need to assume anything at all about the edges of the hypergraph! The (confusing, but extremely useful) property $(c)$ says that the upset generated by $\cG_T$ contains $\cH$, but does not contain any $I \in \cI(\cH)$ such that $f(I) = T$, and that for every set $S \subset V(\cH)$ that is not in $\cG_T$, conditioning a $q$-random set $V_q \subset V(\cH)$ to be independent in $\cG_T$ barely affects the probability that $S$ is contained in $V_q$. 
  
Aragão, Campos, Dahia, Filipe and Marciano applied this lemma to the (highly non-uniform) hypergraph that encodes sets of vertices that induce $(p,R)$-Janson hypergraphs. This allows them to cover the non-$(p,R)$-Janson sets by the independent sets of the `container hypergraphs' $\cG_T$, which encode all of the local obstructions. They then use another (more classical) hypergraph container lemma to study the independent sets of each container hypergraph. This approach seems to be very general and powerful, and we expect to see it used in several further breakthroughs over the coming years.

\end{document}